\documentclass[a4paper,12pt,reqno]{amsart}
\usepackage{amssymb,amsmath,hyperref}
\title[On new tenth-order like identities]{On Ramanujan's lost notebook and new tenth-order like identities for second-, sixth-, and eighth-order mock theta functions}

\author{Eric T. Mortenson}
\address{Department of Mathematics and Computer Science, Saint Petersburg State University, Saint Petersburg, 199034, Russia}
\email{etmortenson@gmail.com}

\renewcommand\theta{\vartheta}

\newtheorem{theorem}{Theorem}
\newtheorem{lemma}[theorem]{Lemma}
\newtheorem{corollary}[theorem]{Corollary}
\newtheorem{proposition}[theorem]{Proposition}
\theoremstyle{definition}

\numberwithin{theorem}{section} 
\numberwithin{equation}{section}

\makeatletter
\@namedef{subjclassname@2020}{\textup{2020} Mathematics Subject Classification}
\makeatother

\begin{document}

\date{21 November 2023}

\subjclass[2020]{11F11, 11F27, 11F37}

\keywords{Appell functions, theta functions, mock theta functions, Ramanujan's lost notebook}

\begin{abstract}
    Ramanujan's lost notebook contains many mock theta functions and mock theta function identities not mentioned in his last letter to Hardy.  For example, we find the four tenth-order mock theta functions and their six identities.  The six identities themselves are of a spectacular nature and were first proved by Choi.  We also find eight sixth-order mock theta functions in the lost notebook, but among their many identities there is only a single relationship like those of the tenth-orders.  Using Appell function properties of Hickerson and Mortenson, we discover and prove three new identities for the sixth-order mock theta functions that are in the spirit of the six tenth-order identities.  We also include an additional nineteen tenth-order like identities for various combinations of second-, sixth-, and eighth-order mock theta functions.
\end{abstract}

\maketitle

\section{Introduction}

 Let $q:=q_{\tau}=e^{2 \pi i \tau}$, $\tau\in\mathbb{H}:=\{ z\in \mathbb{C}| \textup{Im}(z)>0 \}$, and define $\mathbb{C}^*:=\mathbb{C}-\{0\}$.  Recall
\begin{gather*}
(x)_n=(x;q)_n:=\prod_{i=0}^{n-1}(1-q^ix), \ \ (x)_{\infty}=(x;q)_{\infty}:=\prod_{i\ge 0}(1-q^ix),\notag \\
 \Theta(x;q):=(x)_{\infty}(q/x)_{\infty}(q)_{\infty}=\sum_{n=-\infty}^{\infty}(-1)^nq^{\binom{n}{2}}x^n,
\end{gather*}
where the equality between product and sum follows from Jacobi's triple product identity.    Let $a$ and $m$ be integers with $m$ positive.  Define
\begin{gather*}
\Theta_{a,m}:=\Theta(q^a;q^m), \ \ \Theta_m:=\Theta_{m,3m}=\prod_{i\ge 1}(1-q^{mi}), \ {\text{and }}\overline{\Theta}_{a,m}:=\Theta(-q^a;q^m).
\end{gather*}

In Ramanujan's last letter to Hardy, Ramanujan introduced mock theta functions.  He did not give a vigorous definition of mock theta functions; instead, he stated that they have asymptotic properties similar to those of theta functions but that they are not theta functions.  In all, Ramanujan presented seventeen functions divided into groups defined by orders, which were also not well-defined.  One finds four `3rd' order mock theta functions and several identities; ten `5th' order functions and identities; and three `7th' order functions, but with the statement that the `7th' order functions are not related.

Later, in the lost notebook \cite{RLN}, more mock theta functions and mock theta function identities were discovered.  Here we find first mention of the ten identities for the fifth-order mock theta functions which became known as the mock theta conjectures.  These conjectures were subsequently resolved by Hickerson \cite{H1}.  The four tenth-order mock theta functions as well as their six identities made their debut \cite{C1, C2, C3}.  The sixth-order mock theta functions and their identities also made their first appearance \cite{AH, BC}.  Although Hickerson also discovered and proved identities for the seventh-order functions that were analogous to the mock theta conjectures \cite{H2}, the seventh-order functions were conspicuously absent from the lost manuscript.  It is believed that pages from the lost manuscript were likely lost \cite[p. 287]{ABV}.

We begin by revisiting the tenth-order mock theta functions \cite{C1, C2, C3, RLN}
{\allowdisplaybreaks \begin{align*}
{\phi}_{10}(q)&:=\sum_{n\ge 0}\frac{q^{\binom{n+1}{2}}}{(q;q^2)_{n+1}}, \ \ {\psi}_{10}(q):=\sum_{n\ge 0}\frac{q^{\binom{n+2}{2}}}{(q;q^2)_{n+1}}, \\\ 
& \ \ \ \ \ {X}_{10}(q):=\sum_{n\ge 0}\frac{(-1)^nq^{n^2}}{(-q;q)_{2n}}, \ \  {\chi}_{10}(q):=\sum_{n\ge 0}\frac{(-1)^nq^{(n+1)^2}}{(-q;q)_{2n+1}},\notag
\end{align*}}%
which satisfy the following six slightly-rewritten identities.  Letting $\omega$ be a primitive third root of unity, we have \cite{C1,C2, RLN}
{\allowdisplaybreaks \begin{align}
q^{2}\phi_{10}(q^9)-\frac{\psi_{10}(\omega q)-\psi_{10}(\omega^2 q)}{\omega - \omega^2}
&=-q\frac{\Theta_{1,2}}{\Theta_{3,6}}\frac{\Theta_{3,15}\Theta_{6}}{\Theta_{3}},
\label{equation:tenth-id-1}\\
q^{-2}\psi_{10}(q^9)+\frac{\omega \phi_{10}(\omega q)-\omega^2\phi_{10}(\omega^2 q)}{\omega - \omega^2}
&=\frac{\Theta_{1,2}}{\Theta_{3,6}}\frac{\Theta_{6,15}\Theta_{6}}{\Theta_{3}},
\label{equation:tenth-id-2}\\
X_{10}(q^9)-\frac{\omega \chi_{10}(\omega q)-\omega^2\chi_{10}(\omega^2 q)}{\omega - \omega^2}
&=\frac{\overline{\Theta}_{1,4}}{\overline{\Theta}_{3,12}}\frac{\Theta_{18,30}\Theta_{3}}{\Theta_{6}},
\label{equation:tenth-id-3}\\
\chi_{10}(q^9)+q^{2}\frac{ X_{10}(\omega q)-X_{10}(\omega^2 q)}{\omega - \omega^2}&=-q^3\frac{\overline{\Theta}_{1,4}}{\overline{\Theta}_{3,12}}
\frac{\Theta_{6,30}\Theta_{3}}{\Theta_{6}},
\label{equation:tenth-id-4}
\end{align}}%
and \cite{C3, RLN}
\begin{align}
\phi_{10}(q)-q^{-1}\psi_{10}(-q^4)+q^{-2}\chi_{10}(q^8)&=\frac{\overline{\Theta}_{1,2}\Theta(-q^2;-q^{10})}{\Theta_{2,8}},
\label{equation:RLN-id-five}\\
\psi_{10}(q)+q\phi_{10}(-q^4)+X_{10}(q^8)&=\frac{\overline{\Theta}_{1,2}\Theta(-q^6;-q^{10})}{\Theta_{2,8}}.
\label{equation:RLN-id-six}
\end{align}

The six identities were originally found in Ramanujan's lost notebook \cite{RLN}.  What led Ramanujan to these identities is a continuing mystery.  Indeed, in Andrews and Berndt's fifth volume on Ramanujan's lost notebook \cite[p. 396]{ABV}, they state 

\smallskip
``{\em It is inconceivable that an identity such as (\ref{equation:RLN-id-five}) could be stumbled upon by a mindless search algorithm without any overarching theoretical insight.}''   

\smallskip
\noindent The six identities were first proved by Choi \cite{C1, C2, C3} using methods similar to those of Hickerson in his proof of the mock theta conjectures \cite{H1, H2}.  Identities (\ref{equation:tenth-id-1})--(\ref{equation:tenth-id-4}) were later given short proofs by Zwegers \cite{Zw3}.  

We recall that Appell functions are building blocks for Ramanujan's classical mock theta functions.  We will define them as follows
\begin{equation}
m(x,z;q):=\frac{1}{\Theta(z;q)}\sum_{r=-\infty}^{\infty}\frac{(-1)^rq^{\binom{r}{2}}z^r}{1-q^{r-1}xz}.\label{equation:mdef-eq}
\end{equation}
In terms of Appell functions, two of Ramanujan's sixth-order mock theta functions read \cite{AH}, \cite[Section 5]{HM}
\begin{align}
\phi(q)&:=\sum_{n\ge 0}\frac{(-1)^nq^{n^2}(q;q^2)_n}{(-q)_{2n}}=2m(q,-1;q^3)
\label{equation:6th-phi(q)},\\
\psi(q)&:=\sum_{n\ge 0}\frac{(-1)^nq^{(n+1)^2}(q;q^2)_n}{(-q)_{2n+1}}=m(1,-q;q^3)
\label{equation:6th-psi(q)}.
\end{align}

In \cite{M2018}, Mortenson gave short proofs of all six of Ramanujan's identities for the tenth-order mock theta functions by using a recent result on Appell function properties.  

\begin{theorem} \label{theorem:msplit-general-n} \cite[Theorem $3.5$]{HM} For generic $x,z,z'\in \mathbb{C}^*$ 
{\allowdisplaybreaks \begin{align}
 D_n(x,z,z';q)=z' \Theta_n^3  \sum_{r=0}^{n-1}
\frac{q^{{\binom{r}{2}}} (-xz)^r
\Theta\big(-q^{{\binom{n}{2}+r}} (-x)^n z z';q^n\big)
\Theta(q^{nr} z^n/z';q^{n^2})}
{\Theta(xz;q) \Theta(z';q^{n^2}) \Theta\big(-q^{{\binom{n}{2}}} (-x)^n z';q^n)\Theta(q^r z;q^n\big )},
\end{align}}%
where
\begin{equation}
D_n(x,z,z';q):=m(x,z;q) - \sum_{r=0}^{n-1} q^{{-\binom{r+1}{2}}} (-x)^r m\big({-}q^{{\binom{n}{2}-nr}} (-x)^n, z'; q^{n^2} \big).
\label{equation:Dn-def}
\end{equation}
\end{theorem}
\noindent The idea behind the proofs is straightforward.  Once one has the Appell function forms of the tenth-order mock theta functions, one regroups the Appell functions by using (\ref{equation:Dn-def}) and then replaces them with the appropriate sums of quotients of theta functions given by Theorem \ref{theorem:msplit-general-n}.  Each of the six identities is then reduced to proving a theta function identity which can be verified through several applications of the three-term Weierstrass relation for theta functions \cite[(1)]{We}, \cite{Ko}:  For generic $a,b,c,d\in \mathbb{C}^*$
\begin{align}
\Theta(ac;q)\Theta(a/c;q)\Theta(bd;q)\Theta(b/d;q)&=\Theta(ad;q)\Theta(a/d;q)\Theta(bc;q)\Theta(b/c;q)\label{equation:3termWeier}\\
&\qquad +b/c \cdot \Theta(ab;q)\Theta(a/b;q)\Theta(cd;q)\Theta(c/d;q).\notag
\end{align}

There are many identities for the sixth-order functions in the lost notebook \cite{AH, BC}.  Here are two \cite[p. 135, Entry 7.4.2]{ABV}, \cite[p. 13, equations 5b, 6b]{RLN}
\begin{gather}
\phi(q^9)-\psi(q)-q^{-3}\psi(q^9)=\frac{\overline{\Theta}_{3,12}\Theta_{6}^2}{\overline{\Theta}_{1,4}\overline{\Theta}_{9,36}}
\label{equation:RLN6-A},\\
\frac{\psi(\omega q)-\psi(\omega^2q)}{(\omega-\omega^2)q}=\frac{\overline{\Theta}_{1,4}\overline{\Theta}_{9,36}\Theta_{3,6}}{\overline{\Theta}_{3,12}\Theta_{6}}
\label{equation:RLN6-B}.
\end{gather}
Whereas the latter follows from the Appell function property \cite{HM, Zw2}
\begin{equation}
m(x,z_1;q)-m(x,z_0;q)=\frac{z_0(q)_{\infty}^3\Theta(z_1/z_0;q)\Theta(xz_0z_1;q)}{\Theta(z_0;q)\Theta(z_1;q)\Theta(xz_0;q)\Theta(xz_1;q)}, \label{equation:changing-z}
\end{equation}
the former is reminiscent of the six tenth-order identities (\ref{equation:tenth-id-1})-(\ref{equation:RLN-id-six}).  In Section \ref{section:id0}, we will demonstrate that Theorem \ref{theorem:msplit-general-n} can also be used to prove (\ref{equation:RLN6-A}); however, verifying the resulting theta function identity is more difficult, and instead of standard theta function identities, we will use a Maple software package developed by Frank Garvan \cite{FG}.

Of course, there are many more sixth-order mock theta functions in the lost notebook \cite{AH, BC},  see Section \ref{section:additional-results} for a list.  It is natural to ask if they too enjoy identities similar to (\ref{equation:RLN6-A}).   Once one has the Appell function forms of the other sixth-order mock theta functions, see Section \ref{section:additional-results}, one can use Appell function properties and Theorem \ref{theorem:msplit-general-n} to construct identities where one side looks like the left-hand side of (\ref{equation:RLN6-A}) and the other side is a sum of quotients of theta functions.  One sees this play out in Section \ref{section:id0}.  But do the sums collapse to a single quotient?  This leads us to three new identities for the sixth-order mock theta functions.

\begin{theorem} \label{theorem:main} The following identities for the sixth-order mock theta functions $\rho(q)$, $\sigma(q)$, $\lambda(q)$, $\mu(q)$, $\phi_{\_}(q)$, and $\psi_{\_}(q)$ are true
\begin{align}
q\rho(q)+q^3\rho(q^9)-2\sigma(q^9)
&=q\frac{\Theta_{3,6}\Theta_{3}^2}{\Theta_{1,2}\Theta_{9,18}},
\label{equation:newSixth-1}\\
q\lambda(q)+q^{3}\lambda(q^9)-2\mu(q^{9})
&=-\frac{\Theta_{3,6}\Theta_{6}^2}{\overline{\Theta}_{1,4}\overline{\Theta}_{9,36}},
\label{equation:newSixth-2}\\
\psi\_(q)+q^{-3}\psi\_(q^9)-\phi\_(q^9)
&=q\frac{\overline{\Theta}_{3,12}\Theta_{3}^2}{\Theta_{1,2}\Theta_{9,18}}.
\label{equation:newSixth-3}
\end{align}
\end{theorem}
\noindent The sixth-order mock theta functions $\rho(q)$, $\sigma(q)$, $\lambda(q)$, $\mu(q)$, $\phi_{\_}(q)$ are all found in the lost manuscript, so one could ask why are identities (\ref{equation:newSixth-1})-(\ref{equation:newSixth-3}) absent?   We point out that although $\phi_{\_}(q)$, and $\psi_{\_}(q)$ were discovered by Berndt and Chan in \cite{BC}, one also finds $\phi_{\_}(q)$ in \cite[pp. 6, 16]{RLN}.

In Section \ref{section:additional-results}, we state an additional nineteen tenth-order like identities for second-, sixth-, and eighth-order mock theta functions; we also recall the relevant mock theta functions and Appell function forms.  In Section \ref{section:prelim}, we recall basic facts about theta functions and Appell functions.  In Sections \ref{section:tech3} and \ref{section:tech2}, we prove theta function identities using Frank Garvan's Maple packages for $q$-series and theta functions \cite{FG}. In Section \ref{section:id0} we prove identity (\ref{equation:RLN6-A}). In Sections \ref{section:id1} to \ref{section:id3} we prove identities (\ref{equation:newSixth-1})  to (\ref{equation:newSixth-3}) respectively.  In Section \ref{section:newTheorems} we prove the additional nineteen tenth-order like identities for second-, sixth-, and eighth-order mock theta functions.

\section{Mock theta functions and additional results} \label{section:additional-results} 

We have also discovered more mock theta function identities in the spirit of our Theorem \ref{theorem:main}, but they instead follow from the $n=2$ specialization of Theorem \ref{theorem:msplit-general-n}, where $x=1$.   Before we state the new identities, we recall three second-order mock theta functions, the remaining sixth-order mock theta functions, two eighth-order mock theta functions, and two miscellaneous mock theta functions from Ramanujan's lost notebook.  All but the two miscellaneous mock theta functions can be found in \cite[Section 5]{HM}.  We omit the sixth-order mock theta function $\gamma(q)$.

\smallskip
\noindent {\bf `second-order' functions}
{\allowdisplaybreaks \begin{align}
A_2(q)&:=\sum_{n\ge 0}\frac{q^{n+1}(-q^2;q^2)_n}{(q;q^2)_{n+1}}=\sum_{n\ge 0}\frac{q^{(n+1)^2}(-q;q^2)_n}{(q;q^2)_{n+1}^2}=-m(q,q^2;q^4)\label{equation:2nd-A(q)}\\
B_2(q)&:=\sum_{n\ge 0}\frac{q^{n}(-q;q^2)_n}{(q;q^2)_{n+1}}=\sum_{n\ge 0}\frac{q^{n^2+n}(-q^2;q^2)_n}{(q;q^2)_{n+1}^2}=-q^{-1}m(1,q^3;q^4)\label{equation:2nd-B(q)}\\
 \mu_2(q)&:=\sum_{n\ge 0}\frac{(-1)^nq^{n^2}(q;q^2)_n}{(-q^2;q^2)_{n}^2}=2m(-q,-1;q^4)+2m(-q,q;q^4)
\label{equation:2nd-mu(q)}\\
&\ =4m(-q,-1;q^4)-\frac{\Theta_{2,4}^4}{\Theta_1^3}\notag
\end{align}}%

\smallskip
\noindent {\bf `sixth-order' functions}
{\allowdisplaybreaks \begin{align}
\rho(q)&:=\sum_{n\ge 0}\frac{q^{\binom{n+1}{2}}(-q)_n}{(q;q^2)_{n+1}}=-q^{-1}m(1,q;q^6)\label{equation:6th-rho(q)}\\
\sigma(q)&:=\sum_{n\ge 0}\frac{q^{\binom{n+2}{2}}(-q)_n}{(q;q^2)_{n+1}}=-m(q^2,q;q^6)\label{equation:6th-sigma(q)}\\
\lambda(q)&:=\sum_{n\ge 0}\frac{(-1)^nq^n(q;q^2)_n}{(-q)_n}
=q^{-1}m(1,-q^2;q^6)+q^{-1}m(1,-q;q^6)\label{equation:6th-lambda(q)}\\
&\ =2q^{-1}m(1,-q^2;q^6)+\frac{\Theta_{1,2}\overline{\Theta}_{3,12}}{\overline{\Theta}_{1,4}}\notag\\
\mu(q)&:={\sum_{n\ge 0}}^*\frac{(-1)^n(q;q^2)_n}{(-q)_n}=\frac12 +\frac12 \sum_{n\ge 0}\frac{(-1)^nq^{n+1}(1+q^n)(q;q^2)_n}{(-q;q)_{n+1}}\label{equation:6th-mu(q)}\\
&\ =m(q^2,-1;q^6)+ m(q^2,-q^3;q^6)=2m(q^2,-1;q^6)-\frac{\Theta_{1,2}\overline{\Theta}_{1,3}}{2\overline{\Theta}_{1,4}}\notag\\
 {\phi}\_(q)&:=\sum_{n\ge 1}\frac{q^n(-q;q)_{2n-1}}{(q;q^2)_n}=-\frac{3}{4}m(q,q;q^3)-\frac{1}{4}m(q,-q;q^3)
\label{equation:phibar}\\
&\ =-m(q,q;q^3)-q\frac{\overline{\Theta}_{3,12}^3}{\Theta_1\overline{\Theta}_{1,4}}\notag\\
{\psi}\_(q)&:=\sum_{n\ge 1}\frac{q^n(-q;q)_{2n-2}}{(q;q^2)_n}=-\frac{3}{4}m(1,q;q^3)+\frac{1}{4}m(1,-q;q^3)\label{equation:psibar}\\
&\ =-\frac{1}{2}m(1,q;q^3)+q\frac{{\Theta}_{6}^3}{2\Theta_1 \Theta_2}\notag
\end{align}}%

\smallskip
\noindent {\bf `eighth-order' functions}
{\allowdisplaybreaks \begin{align}
U_0(q)&:=\sum_{n\ge 0}\frac{q^{n^2}(-q;q^2)_n}{(-q^4;q^4)_{n}}=2m(-q,-1;q^4)\label{equation:8th-U0(q)}\\
U_1(q)&:=\sum_{n\ge 0}\frac{q^{(n+1)^2}(-q;q^2)_n}{(-q^2;q^4)_{n+1}}=-m(-q,-q^2;q^4)\label{equation:8th-U1(q)}
\end{align}}

\smallskip
\noindent {\bf miscellaneous functions}
\begin{align}
  \phi_R(q)&:=\sum_{n\ge0}\frac{q^{n+1}(-q)_{2n}}{(q;q^2)_{n+1}^2}=-\frac{1}{2}m(1,q;q^2)
\label{equation:misc-RLNp3}\\
\xi_{R}(q)&:=\sum_{n\ge0}{}^{*}\frac{(-1)^n(q;q^2)_n}{(-q)_{n}^2}
=\frac{1}{2}+\frac{1}{2}\sum_{n\ge0}\frac{(-1)^{n}q^{n+1}(q;q^2)_n(2+q^n+q^{n+1})}{(-q)_{n+1}^2}
\label{equation:misc-RLNp4}\\
& \ =2m(1,-1;q)\notag
\end{align}

The reader will note that the sums in (\ref{equation:6th-mu(q)}) and (\ref{equation:misc-RLNp4}) do not converge.  However, for each sum, the sequence of even partial sums and the sequence of odd partial sums both converge.  We define $\sum^{\star}$ to be the average of the two limits, see for example \cite[(0.16)]{AH}.

The second-order mock theta functions $A_2(q)$ and $\mu_2(q)$ are found in \cite[p. 8]{RLN} and \cite[pp. 8, 29]{RLN} respectively.   As is pointed out in \cite{Mc07}, $B_2(q)$ is related to $A_2(q)$ and $\mu_2(q)$ through modularity.  The locations of the sixth-order mock theta from the lost notebook are detailed in \cite{AH}.   The sixth-order mock theta functions $\phi_{-}(q)$ and $\psi_{-}(q)$ were discovered in \cite{BC}, but one can also find $\phi_{-}(q)$ in \cite[pp. 6, 16]{RLN}.  The two eighth-order mock theta functions are found in \cite{GM}.  The two miscellaneous mock theta functions $\phi_R(q)$ and $\xi_R(q)$ are found in \cite[p. 3]{RLN} and \cite[p. 4]{RLN} respectively.  The functions $\phi_R(q)$ and $\xi_R(q)$ also appear  as the mock modular forms $H_2^{(4)}$ and $H_1^{(2)}$ of type $2A$ in \cite{CDH}.

To the best of our knowledge, the following nineteen identities are new.

\begin{theorem}\label{theorem:3B-ids} We have
{\allowdisplaybreaks \begin{align}
qB_2(q)-2A_2(-q^4)&=q \frac{\Theta_{2}}{\Theta_{1}^2}\frac{\Theta_{4}^5\Theta_{16}^2}{\Theta_{8}^5},
\label{equation:Bid-1}\\
qB_2(q)+\frac{1}{2}\mu_2(q^4)&=\frac{1}{2}\frac{\Theta_{2}}{\Theta_{1}^2}\frac{\Theta_{4}^3\Theta_{8}}{\Theta_{16}^2},
\label{equation:Bid-2}\\
qB_2(q)+\frac{1}{4}\mu_2(q^4)-A_2(-q^4)&=\frac{1}{4}\frac{\Theta_{2}^6}{\Theta_{1}^{4}}\frac{\Theta_{4}^3}{\Theta_{8}^4}.
\label{equation:Bid-3}
\end{align}}%
\end{theorem}

\begin{theorem}\label{theorem:3rho-ids} We have
{\allowdisplaybreaks \begin{align}
q\rho(q)-2A_2(-q^6)&=q \frac{\Theta_{2}^2\Theta_{3}^2\Theta_{6}\Theta_{8}\Theta_{24}}
{\Theta_{1}^2\Theta_{4}\Theta_{12}^3},
\label{equation:rho-id-1}\\
q\rho(q)+\frac{1}{2}\mu_2(q^6)
&=\frac{1}{2}\frac{\Theta_{2}\Theta_{3}^2\Theta_{4}^2}
{\Theta_{1}^2\Theta_{8}\Theta_{24}},
\label{equation:rho-id-2}\\
q\rho(q)+\frac{1}{4}\mu_2(q^6)-A_2(-q^6)&=\frac{1}{4}\frac{\Theta_{2}^6\Theta_{3}^4}
{\Theta_{1}^4\Theta_{4}^2\Theta_{6}\Theta_{12}^2}.
\label{equation:rho-id-3}
\end{align}}%
\end{theorem}

\begin{theorem}\label{theorem:3lambda-ids} We have
{\allowdisplaybreaks \begin{align}
\frac{q}{2}\lambda(q)+2A_2(-q^6)&=\frac{q}{2}\frac{\Theta_{1}\Theta_{3}\Theta_{4}^5\Theta_{6}^3\Theta_{24}^{2}}
{\Theta_{2}^4\Theta_{8}^{2}\Theta_{12}^5},
\label{equation:lambda-id-1}\\
\frac{q}{2}\lambda(q)-\frac{1}{2}\mu_2(q^6)&=-\frac{1}{2}\frac{\Theta_{1}\Theta_{3}\Theta_{6}\Theta_{8}^2\Theta_{12}}
{\Theta_{2}^2\Theta_{4}\Theta_{24}^2},
\label{equation:lambda-id-2}\\
\frac{q}{2}\lambda(q)-\frac{1}{4}\mu_2(q^6)+A_2(-q^6)&=-\frac{1}{4}\frac{\Theta_{1}^2\Theta_{3}^2\Theta_{4}\Theta_{6}^2}
{\Theta_{2}^3\Theta_{12}^3}.
\label{equation:lambda-id-3}
\end{align}}%
\end{theorem}

\begin{theorem}\label{theorem:3Rphi-ids} We have
{\allowdisplaybreaks \begin{align}
2\phi_{R}(q)-2A_{2}(-q^2)
&=2q\frac{\Theta_{2}^7\Theta_{8}^4}{\Theta_{1}^4\Theta_{4}^6},
\label{equation:Rphi-id-1}\\
2\phi_{R}(q)+\frac{1}{2}\mu_{2}(q^2)
&=\frac{1}{2}\frac{\Theta_{2}^3\Theta_{4}^6}{\Theta_{1}^4\Theta_{8}^4},
\label{equation:Rphi-id-2}\\
2\phi_{R}(q)+\frac{1}{4}\mu_{2}(q^2)-A_{2}(-q^2)
&=\frac{1}{4}\frac{\Theta_{2}^{17}}{\Theta_{1}^8\Theta_{4}^8}.
\label{equation:Rphi-id-3}
\end{align}}%
\end{theorem}

\begin{theorem}\label{theorem:3Rxi-ids} We have
{\allowdisplaybreaks \begin{align}
\frac{1}{2}\xi_{R}(q)+2A_2(-q)&=\frac{1}{4}\frac{\Theta_{1}^5}{\Theta_{2}^4},
\label{equation:3xi-id-1}\\
\frac{1}{2}\xi_{R}(q)-\frac{1}{2}\mu_2(q)&=-\frac{1}{4}\frac{\Theta_{1}^5}{\Theta_{2}^4},
\label{equation:3xi-id-2}\\
\frac{1}{2}\xi_{R}(q)-\frac{1}{4}\mu_2(q)+A_2(-q)&=0.
\label{equation:3xi-id-3}
\end{align}}%
\end{theorem}

\begin{theorem} \label{theorem:4psi-ids} We have
{\allowdisplaybreaks \begin{align}
\psi(q)-U_{0}(q^3)&=-\frac{\Theta_{1}\Theta_{6}^4\Theta_{8}^2\Theta_{12}}
{\Theta_{2}^2\Theta_{3}^2\Theta_{4}\Theta_{24}^2},
\label{equation:psi-id-1}\\
\psi(q)+2U_{1}(q^3)&=q \frac{\Theta_{1}\Theta_{4}^5\Theta_{6}^6\Theta_{24}^2}
{\Theta_{2}^4\Theta_{3}^2\Theta_{8}^2\Theta_{12}^5},
\label{equation:psi-id-2}\\
2\psi_{\_}(q)+U_0(q^3)&=
\frac{\Theta_{2}\Theta_{4}^{2}\Theta_{6}^3}{\Theta_{1}^2\Theta_{3}\Theta_{8}\Theta_{24}},
\label{equation:psiBar-id-1}\\
2\psi_{\_}(q)-2U_1(q^3)&=2q\frac{\Theta_2^2\Theta_6^4\Theta_8\Theta_{24}}{\Theta_{1}^2\Theta_{3}\Theta_{4}\Theta_{12}^3}.
\label{equation:psiBar-id-2}
\end{align}}%
\end{theorem}

\section{Preliminaries}\label{section:prelim}

We will frequently use the following identities without mention.  They easily follow from the definitions.
{\allowdisplaybreaks \begin{subequations}
\begin{gather}
\overline{\Theta}_{0,1}=2\overline{\Theta}_{1,4}=\frac{2\Theta_2^2}{\Theta_1},  \ 
\overline{\Theta}_{1,2}=\frac{\Theta_2^5}{\Theta_1^2\Theta_4^2}, \ 
  \Theta_{1,2}=\frac{\Theta_1^2}{\Theta_2},   \ \overline{\Theta}_{1,3}=\frac{\Theta_2\Theta_3^2}{\Theta_1\Theta_6},
   \notag\\
\Theta_{1,4}=\frac{\Theta_1\Theta_4}{\Theta_2},  
\  \Theta_{1,6}=\frac{\Theta_1\Theta_6^2}{\Theta_2\Theta_3},   \ 
\overline{\Theta}_{1,6}=\frac{\Theta_2^2\Theta_3\Theta_{12}}{\Theta_1\Theta_4\Theta_6}.\notag
\end{gather}
\end{subequations}}%

We have the general identities:
\begin{subequations}
{\allowdisplaybreaks \begin{gather}
\Theta(q^n x;q)=(-1)^nq^{-\binom{n}{2}}x^{-n}\Theta(x;q), \ \ n\in\mathbb{Z},\label{equation:theta-elliptic}\\
\Theta(x;q)=\Theta(q/x;q)=-x\Theta(x^{-1};q)\label{equation:theta-inv},\\
\Theta(x;q)={\Theta_1}\Theta(x;q^2)\Theta(qx;q^2)/{\Theta_2^2}, \label{equation:theta-mod}\\
\Theta(z;q)=\Theta(-z^2q;q^4)-z\Theta(-z^2q^3;q^4),\label{equation:theta-split}\\
\Theta(x^2;q^2)={\Theta_2}\Theta(x;q)\Theta(-x;q)/{\Theta_1^2}. \label{equation:theta-roots}
\end{gather}}%
\end{subequations}

We also have the following \cite[Theorem 1.1]{H1}:  For generic $x,y\in \mathbb{C}^*$ 
\begin{equation}
\Theta(x;q)\Theta(y;q)=\Theta(-xy;q^2)\Theta(-qx^{-1}y;q^2)-x \Theta(-qxy;q^2) \Theta(-x^{-1}y;q^2).
\label{equation:H1Thm1.1}
\end{equation}

The Appell function $m(x,q,z)$ satisfies several functional equations and identities, which we collect in the form of a proposition \cite{HM, Zw2}:  

\begin{proposition}  For generic $x,z\in \mathbb{C}^*$
{\allowdisplaybreaks \begin{subequations}
\begin{gather}
m(x,z;q)=m(x,qz;q),\label{equation:mxqz-fnq-z}\\
m(x,z;q)=x^{-1}m(x^{-1},z^{-1};q),\label{equation:mxqz-flip}\\
m(x,z;q)=m(x,x^{-1}z^{-1};q).\label{equation:mxqz-fnq-newz}
\end{gather}
\end{subequations}}
\end{proposition}

We point out the  $n=2$ and $n=3$ specializations of  Theorem \ref{theorem:msplit-general-n}:

\begin{corollary} \label{corollary:msplitn2zprime} For generic $x,z,z'\in \mathbb{C}^*$ 
{\allowdisplaybreaks \begin{align}
D_2&(x,z,z';q) \label{equation:msplit2}\\
&=\frac{z'(q^2;q^2)_{\infty}^3}{\Theta(xz;q)\Theta(z';q^4)}\Big [
\frac{\Theta(-qx^2zz';q^2)\Theta(z^2/z';q^{4})}{\Theta(-qx^2z';q^2)\Theta(z;q^2)}
-xz \frac{\Theta(-q^2x^2zz';q^2)\Theta(q^2z^2/z';q^{4})}{\Theta(-qx^2z';q^2)\Theta(qz;q^2)}\Big ],\notag
\end{align}}%
where
\begin{equation}
D_2(x,z,z';q):=m(x,z;q)-m(-qx^2,z';q^4 )+q^{-1}xm(-q^{-1}x^2,z';q^4).\label{equation:D2-def}
\end{equation}
\end{corollary}

\begin{corollary} \label{corollary:msplitn3zprime} For generic $x,z,z'\in \mathbb{C}^*$ 
\begin{align}
D_3(x,q,z,z')&=\frac{z'\Theta_3^3}{\Theta(xz;q)\Theta(z';q^{9})\Theta(x^3z';q^3)}\Big [ 
\frac{1}{z}\frac{\Theta(x^3zz';q^3)\Theta(z^3/z';q^{9})}{\Theta(z;q^3)}\label{equation:msplit3} \\
&\ \ \ \ \ -\frac{x}{q}\frac{\Theta(qx^3zz';q^3)\Theta(q^{3}z^3/z';q^{9})}{\Theta(qz;q^3)}
+\frac{x^2z}{q}\frac{\Theta(q^2x^3zz';q^3)\Theta(q^{6}z^3/z';q^{9})}{\Theta(q^2z;q^3)}\Big ],\notag
\end{align}
where
\begin{align}
D_3(x,z,z';q)&:=m(x,z;q)-m\Big (q^{3}x^3,z';q^{9}\Big )\label{equation:D3-def}\\
&\ \ \ \ \  +q^{-1}xm\Big (x^3,z';q^{9}\Big )-q^{-3}x^2m\Big (q^{-3}x^3,z';q^{9}\Big ).\notag
\end{align}
\end{corollary}

\section{Technical Results for $n=3$}\label{section:tech3}
\begin{lemma} \label{lemma:dTerm-id1} We have
{\allowdisplaybreaks \begin{align}
D_3(1,-q,-q^9;q^3)
&=\frac{\Theta_{9}^3}{\overline{\Theta}_{1,3}\overline{\Theta}_{9,27}\overline{\Theta}_{0,9}}
\Big [ 
q\frac{\Theta_{1,9}\Theta_{6,27}}{\overline{\Theta}_{1,9}}
-q^{2}\frac{\Theta_{4,9}\Theta_{3,27}}{\overline{\Theta}_{4,9}}
-\frac{\Theta_{7,9}\Theta_{12,27}}{\overline{\Theta}_{7,9}} \Big ],
\label{equation:dTerm-id1}\\
D_3(1,q,q^9;q^6)
&= 
-\frac{\Theta_{18}^3}{\Theta_{1,6}\Theta_{9,54}\Theta_{9,18}}
\Big [ q^{2} \frac{\Theta_{10,18}\Theta_{6,54}}{\Theta_{1,18}}
 +q^{3} \frac{\Theta_{16,18}\Theta_{12,54}}{\Theta_{7,18}}
+\frac{\Theta_{4,18}\Theta_{30,54}}{\Theta_{13,18}}  \Big ],
\label{equation:dTerm-id2}\\
D_3(1,-q^2,-q^{27};q^6)
&=\frac{\Theta_{18}^3}{\overline{\Theta}_{2,6}\overline{\Theta}_{27,54}\overline{\Theta}_{9,18}}
 \Big [ q^{2}\frac{\Theta_{11,18}\Theta_{21,54}}{\overline{\Theta}_{2,18}}
+q^{10}\frac{\Theta_{17,18}\Theta_{3,54}}{\overline{\Theta}_{8,18}}
+q^{4}\frac{\Theta_{5,18}\Theta_{15,54}}{\overline{\Theta}_{14,18}}
\Big ],
\label{equation:dTerm-id3A}\\
D_3(1,-q,-1;q^6)
&=\frac{\Theta_{18}^3}{\overline{\Theta}_{1,6}\overline{\Theta}_{0,54}\overline{\Theta}_{0,18}}
 \Big [ q^{-1}\frac{\Theta_{1,18}\Theta_{3,54}}{\overline{\Theta}_{1,18}}
+q^{-6}\frac{\Theta_{7,18}\Theta_{21,54}}{\overline{\Theta}_{7,18}}
+q^{-5}\frac{\Theta_{13,18}\Theta_{39,54}}{\overline{\Theta}_{13,18}}
\Big ],
\label{equation:dTerm-id3B}\\
D_{3}(1,q,-1;q^3)
&=-\frac{\Theta_{9}^3}{\Theta_{1}\overline{\Theta}_{0,27}\overline{\Theta}_{0,9}}
 \Big [ q^{-1}\frac{\overline{\Theta}_{1,9}\overline{\Theta}_{3,27}}{\Theta_{1,9}}
-q^{-3}\frac{\overline{\Theta}_{4,9}\overline{\Theta}_{12,27}}{\Theta_{4,9}}
+q^{-2}\frac{\overline{\Theta}_{7,9}\overline{\Theta}_{21,27}}{\Theta_{7,9}}
\Big ]. 
\label{equation:dTerm-id4}
\end{align}}%
\end{lemma}
\begin{proof}[Proof of Lemma \ref{lemma:dTerm-id1}]
We use Corollary \ref{corollary:msplitn3zprime} and (\ref{equation:theta-elliptic}).
\end{proof}

\begin{proposition}\label{proposition:finalTheta-ids}  We have
{\allowdisplaybreaks \begin{align}
\frac{\overline{\Theta}_{3,12}\Theta_{6}^2}{\overline{\Theta}_{1,4}\overline{\Theta}_{9,36}}
&=-D_3(1,-q,-q^9;q^3)
 +\frac{\Theta_{27}^3\Theta_{9}^2}{\overline{\Theta}_{0,27}\overline{\Theta}_{9,27}^3},
\label{equation:finalTheta-id1}\\
q\frac{\Theta_{3,6}\Theta_{3}^2}{\Theta_{1,2}\Theta_{9,18}}
&=-D_3(1,q,q^9;q^6)-\frac{\Theta_{54}^3\Theta_{18,54}^2}{\Theta_{9,54}^3\Theta_{27,54}},
\label{equation:finalTheta-id2}\\
-\frac{\Theta_{3,6}\Theta_{6}^2}{\overline{\Theta}_{1,4}\overline{\Theta}_{9,36}}
&=D_3(1,-q^2,-q^{27};q^6)+D_3(1,-q,-1;q^6) \label{equation:finalTheta-id3} \\
&\qquad \qquad +q^{12}\frac{\Theta_{54}^3\Theta_{9,54}^2}{\overline{\Theta}_{27,54}^2\overline{\Theta}_{18,54}^2}
-q^{-6}\frac{\Theta_{54}^3\Theta_{9,54}^2}{\overline{\Theta}_{0,54}^2\overline{\Theta}_{9,54}^2},
\notag \\
2q\frac{\overline{\Theta}_{3,12}\Theta_{3}^2}{\Theta_{1,2}\Theta_{9,18}}
&=-D_{3}(1,q,-1;q^3)+q\frac{\Theta_{6}^3}{\Theta_{1}\Theta_{2}}
+2q^{9}\frac{\overline{\Theta}_{27,108}^3}{\Theta_{9}\overline{\Theta}_{9,36}}\label{equation:finalTheta-id4} \\
&\qquad -2\frac{\Theta_{27}^3\overline{\Theta}_{9,27}^2}{\Theta_{9}^2\overline{\Theta}_{0,27}\overline{\Theta}_{9,27}} +q^{6}\frac{\Theta_{54}^3}{\Theta_{9}\Theta_{18}}
+q^{-3}\frac{\Theta_{27}^3\overline{\Theta}_{9,27}^2}{\Theta_{9}^2\overline{\Theta}_{0,27}^2}.\notag
\end{align}}%
\end{proposition}
\begin{proof}[Proof of Proposition \ref{proposition:finalTheta-ids}]  Frank Garvan's Maple packages {\em qseries} and {\em thetaids} prove all four theta function identities \cite{FG}.  We give a brief description of the process where we use (\ref{equation:finalTheta-id2}) as a running example.  

We first normalize (\ref{equation:finalTheta-id2}) to obtain the equivalent identity
\begin{equation}
g({\tau}):=f_1(\tau)+f_2(\tau)+f_3(\tau)-f_4(\tau)-1=0,\label{equation:finalTheta-id2-normal}
\end{equation}
where
\begin{gather*}
f_1(\tau):=q\frac{\Theta_{1,2}\Theta_{9,18}}{\Theta_{3,6}\Theta_{3}^2}
\frac{\Theta_{18}^3}{\Theta_{1,6}\Theta_{9,54}\Theta_{9,18}}
 \frac{\Theta_{10,18}\Theta_{6,54}}{\Theta_{1,18}},\\
f_2(\tau):=q^2\frac{\Theta_{1,2}\Theta_{9,18}}{\Theta_{3,6}\Theta_{3}^2}
\frac{\Theta_{18}^3}{\Theta_{1,6}\Theta_{9,54}\Theta_{9,18}}
 \frac{\Theta_{16,18}\Theta_{12,54}}{\Theta_{7,18}},\\
f_3(\tau):=\frac{1}{q}\frac{\Theta_{1,2}\Theta_{9,18}}{\Theta_{3,6}\Theta_{3}^2}
\frac{\Theta_{18}^3}{\Theta_{1,6}\Theta_{9,54}\Theta_{9,18}}
\frac{\Theta_{4,18}\Theta_{30,54}}{\Theta_{13,18}}, \quad
f_4(\tau):=\frac{1}{q}\frac{\Theta_{1,2}\Theta_{9,18}}{\Theta_{3,6}\Theta_{3}^2}
\frac{\Theta_{54}^3\Theta_{18,54}^2}{\Theta_{9,54}^3\Theta_{27,54}}.
\end{gather*}


For the first step, one uses \cite[Theorem 18]{Rob} to verify that each $f_{j}(\tau)$ is a modular function on $\Gamma_{1}(54)$ for each $1\le j\le 4$.  For the second step, one uses \cite[Corollary 4]{CKP} to find a set $\mathcal{S}_{54}$ of inequivalent cusps for $\Gamma_{1}(54)$.  One also determines the fan width of each cusp.  For the third step, one uses \cite[Lemma 3.2]{Biag} to calculate the invariant order of each modular function at each of the cusps of $\Gamma_{1}(54)$.  For the fourth step, one uses the Valence Formula \cite[p. 98]{Rank} to determine the number of terms to verify in order to confirm identity (\ref{equation:finalTheta-id2-normal}).  To this end, one calculates $B$ where 
\begin{equation*}
B:=\sum_{\substack{s\in\mathcal{S}_{N}\\ s\ne i\infty}}\textup{min}(\{ \textup{ORD}(f_{j},s,\Gamma_{1}(54)):1\le j\le n\}\cup\{0\}),
\end{equation*}
and where $\textup{ORD}(f,\zeta,\Gamma):=\kappa(\zeta,\Gamma)\textup{ord}(f,\zeta)$, with $\kappa(\zeta,\Gamma)$ denoting the fan width of the cusp $\zeta$ and $\textup{ord}(f,\zeta)$ denoting the invariant order.  We direct the interested reader to \cite[p. 91]{Rank} for further details.  

In our running example, $B=-63$.  From the Valence Formula \cite[Corollary 2.5]{FG} we know that (\ref{equation:finalTheta-id2-normal}) is true if and only if 
\begin{equation*}
\textup{ORD}(g(\tau), i\infty,\Gamma_{1}(54))>-B
\end{equation*}
Hence one only needs to verify identity (\ref{equation:finalTheta-id2-normal}) out through $\mathcal{O}(q^{64})$, which is what one does in the fifth and final step.
\end{proof}

\section{Technical Results for $n=2$}\label{section:tech2}

In the proofs of Theorems \ref{theorem:3B-ids}--\ref{theorem:4psi-ids}, not all of the $D_{2}(1,z,z';q)$ terms encountered evaluate to a single quotient of theta functions.  We list the ones that do.  Some their proofs follow from straightforward evaluations using classical theta function identities and properties, while for other proofs we need to use the methods of \cite{FG}.

\begin{proposition}\label{proposition:3B-ids} We have
\begin{align}
D_{2}(1,q^3,q^8;q^{4})
&=-q\frac{\Theta_{2}\Theta_{4}^5\Theta_{16}^2}{\Theta_{1}^2\Theta_{8}^5},
\label{equation:D2-B-id1}\\
D_{2}(1,q^{3},q^{12};q^{4})
&=-\frac{1}{2}\frac{\Theta_{2}\Theta_{4}^3\Theta_{8}}{\Theta_{1}^2\Theta_{16}^2}.
\label{equation:D2-B-id2}
\end{align}
\end{proposition}

\begin{proposition}\label{proposition:3rho-ids} We have
\begin{align}
D_{2}(1,q,q^{12};q^{6})
&=-q \frac{\Theta_{2}^2\Theta_{3}^2\Theta_{6}\Theta_{8}\Theta_{24}}
{\Theta_{1}^2\Theta_{4}\Theta_{12}^3},
\label{equation:D2-rho-id1}\\
D_2(1,q,q^{18};q^{6})
&=-\frac{1}{2}\frac{\Theta_{2}\Theta_{3}^2\Theta_{4}^2}
{\Theta_{1}^2\Theta_{8}\Theta_{24}}.
\label{equation:D2-rho-id2}
\end{align}
\end{proposition}

\begin{proposition}\label{proposition:3lambda-ids} We have
\begin{align}
D_2(1,-q^2,q^{12};q^{6})
&=q^2\frac{\Theta_{2}\Theta_{6}^2\Theta_{24}^3}{\Theta_{8}\Theta_{12}^4},
\label{equation:D2-lambda-id1}\\
D_2(1,-q,q^{12};q^{6})
&=q\frac{\Theta_{1}^2\Theta_{4}\Theta_{6}^7\Theta_{8}\Theta_{24}}{\Theta_{2}^4\Theta_{3}^2\Theta_{12}^5},
\label{equation:D2-lambda-id2}\\
D_2(1,-q^2,-1;q^{6})
&=-\frac{\Theta_{2}\Theta_{12}^4\Theta_{16}^2\Theta_{24}}{\Theta_{4}^2\Theta_{6}^2\Theta_{8}\Theta_{48}^2}.
\label{equation:D2-lambda-id3}
\end{align}
\end{proposition}

\begin{proposition}\label{proposition:3Rphi-ids} We have
\begin{align}
D_{2}(1,q,q^4;q^2)
&=-2q\frac{\Theta_{2}^7\Theta_{8}^4}{\Theta_{1}^4\Theta_{4}^6},
\label{equation:D2-Rphi-id1}\\
D_{2}(1,q,q^{-2};q^2)
&=-\frac{1}{2}\frac{\Theta_{2}^3\Theta_{4}^6}{\Theta_{1}^4\Theta_{8}^4}.
\label{equation:D2-Rphi-id2}
\end{align}
\end{proposition}

\begin{proposition}\label{proposition:3Rxi-ids} We have
\begin{align}
D_2(1,-1,q^2;q)
&=\frac{1}{4}\frac{\Theta_{1}^5}{\Theta_{2}^4},
\label{equation:D2-Rxi-id1}\\
D_2(1,-1,q^{-1};q)
&=-\frac{1}{4}\frac{\Theta_{1}^5}{\Theta_{2}^4}.
\label{equation:D2-Rxi-id2}
\end{align}
\end{proposition}

\begin{proposition}\label{proposition:4psi-ids} We have
\begin{align}
D_2(1,-q,-1;q^{3})
&=-\frac{\Theta_{1}\Theta_{6}^4\Theta_{8}^2\Theta_{12}}
{\Theta_{2}^2\Theta_{3}^2\Theta_{4}\Theta_{24}^2},
\label{equation:D2-psi-id1}\\
D_2(1,-q,-q^{6};q^3)
&=q \frac{\Theta_{1}\Theta_{4}^5\Theta_{6}^6\Theta_{24}^2}
{\Theta_{2}^4\Theta_{3}^2\Theta_{8}^2\Theta_{12}^5}.
\label{equation:D2-psi-id2}
\end{align}
\end{proposition}

The first lemma is used in the proofs of Propositions \ref{proposition:3lambda-ids} and \ref{proposition:4psi-ids}.  The second lemma is used in the proofs of Theorems \ref{theorem:3rho-ids}--\ref{theorem:3Rphi-ids}, \ref{theorem:4psi-ids}.  The proofs of both lemmas follow that of the proof of Proposition \ref{proposition:finalTheta-ids}, so we will not include them.

\begin{lemma}  We have the following theta function identities
{\allowdisplaybreaks \begin{gather}
\frac{1}{2} \frac{\Theta_{2}\Theta_{8}^5\Theta_{12}^5\Theta_{48}}{\Theta_{4}^3\Theta_{6}^3\Theta_{16}^3\Theta_{24}^3}
 +\frac{1}{2}\frac{\Theta_{4}^4\Theta_{24}^4}{\Theta_{2}\Theta_{6}\Theta_{8}^2\Theta_{12}^2\Theta_{16}\Theta_{48}}
 -1=0,\label{equation:D2-lambda-id3-thetaID}\\
 \frac{1}{2}\frac{\Theta_{1}\Theta_{4}^5\Theta_{6}^5\Theta_{24}}
{\Theta_{2}^3\Theta_{3}^3\Theta_{8}^3\Theta_{12}^3}
+\frac{1}{2}\frac{\Theta_{2}^4\Theta_{12}^4}
{\Theta_{1}\Theta_{3}\Theta_{4}^2\Theta_{6}^2\Theta_{8}\Theta_{24}}
-1=0,\label{equation:equation:D2-psi-id1-thetaID}\\
\frac{\Theta_{1}\Theta_{6}^{4}\Theta_{8}^3}
{\Theta_{3}^3\Theta_{4}^4\Theta_{24}}
+q\frac{\Theta_{2}^5\Theta_{8}\Theta_{12}\Theta_{24}}
{\Theta_{1}\Theta_{3}\Theta_{4}^5\Theta_{6}}
- 1=0.\label{equation:equation:D2-psi-id2-thetaID}
\end{gather}}%
\end{lemma}

\begin{lemma}  We have the following theta function identities
{\allowdisplaybreaks \begin{gather}
\frac{\Theta_{1}^2\Theta_{4}^4\Theta_{6}\Theta_{12}^2}
{\Theta_{2}^5\Theta_{3}^2\Theta_{8}\Theta_{24}}+
2q\frac{\Theta_{1}^2\Theta_{4}\Theta_{6}^2\Theta_{8}\Theta_{24}}
{\Theta_{2}^4\Theta_{3}^2\Theta_{12}}
-1=0,\label{equation:rho-id-3-thetaID}\\
\frac{\Theta_{1}\Theta_{6}^4\Theta_{8}^{3}}
{\Theta_{3}^3\Theta_{4}^4\Theta_{24}}
+q\frac{\Theta_{2}^5\Theta_{8}\Theta_{12}\Theta_{24}}
{\Theta_{1}\Theta_{3}\Theta_{4}^5\Theta_{6}}
-1=0,\label{equation:lambda-id-1-thetaID}\\
\frac{\Theta_{2}^3\Theta_{12}^3\Theta_{16}^2\Theta_{24}^3}
{\Theta_{1}\Theta_{3}\Theta_{4}\Theta_{6}^3\Theta_{8}^3\Theta_{48}^2}
-q\frac{\Theta_{4}^2\Theta_{12}^4\Theta_{48}}
{\Theta_{3}^2\Theta_{6}^2\Theta_{8}^2\Theta_{24}}
\cdot \left ( 
\frac{\Theta_{1,12}\Theta_{20,48}}{\Theta_{5,12}\Theta_{2,24}}
+q^4
\frac{\Theta_{5,12}\Theta_{4,48}}{\Theta_{1,12}\Theta_{10,24}}\right ) 
-1=0, \label{equation:lambda-id-2-thetaID}\\
q\frac{\Theta_{4}^4\Theta_{6}\Theta_{24}^{2}}
{\Theta_{1}\Theta_{2}\Theta_{3}\Theta_{8}^{2}\Theta_{12}^2}
-\frac{\Theta_{2}\Theta_{8}^2\Theta_{12}^4}
{\Theta_{1}\Theta_{3}\Theta_{4}^2\Theta_{6}\Theta_{24}^2}
+1=0,\label{equation:lambda-id-3-thetaID}\\
4q\frac{\Theta_{1}^4\Theta_{4}^2\Theta_{8}^4}{\Theta_{2}^{10}}
+\frac{\Theta_{1}^4\Theta_{4}^{14}}{\Theta_{8}^4\Theta_{2}^{14}}
-1=0,\label{equation:Rphi-id-3-thetaID}\\
\frac{\Theta_{1}^2\Theta_{6}^3\Theta_{8}^2\Theta_{12}^3}
{\Theta_{2}^3\Theta_{3}^2\Theta_{4}^3\Theta_{24}^2}
  +2q\frac{\Theta_{1}\Theta_{3}\Theta_{8}\Theta_{24}}{\Theta_{2}^2\Theta_{4}^{2}} 
- 1=0,\label{equation:psiBar-id-1-thetaID}\\
q\frac{\Theta_{1}^2\Theta_{4}^3\Theta_{6}^5\Theta_{24}^2}
{\Theta_{2}^5\Theta_{3}^2\Theta_{8}^2\Theta_{12}^3}
 +\frac{\Theta_{1}\Theta_{3}\Theta_{4}\Theta_{12}^3}
 {\Theta_2^3\Theta_6\Theta_8\Theta_{24}}
 -1=0.\label{equation:psiBar-id-2-thetaID}
\end{gather}}%
\end{lemma}

\begin{proof}[Proof of Proposition  \ref{proposition:3B-ids}] 
We prove (\ref{equation:D2-B-id1}).  We substitute into (\ref{equation:msplit2}) to obtain
\begin{align*}
D_{2}(1,q^3,q^8;q^{4})&=\frac{q^8(q^8;q^8)_{\infty}^3}{\Theta(q^3;q^4)\Theta(q^8;q^{16})}\\
&\qquad \cdot \Big [
\frac{\Theta(-q^{15};q^8)\Theta(q^{-2};q^{16})}{\Theta(-q^{12};q^8)\Theta(q^3;q^8)}
-q^3 \frac{\Theta(-q^{19};q^8)\Theta(q^{6};q^{16})}{\Theta(-q^{12};q^8)\Theta(q^7;q^8)}\Big ].
\end{align*}
Using (\ref{equation:theta-elliptic}) and (\ref{equation:theta-inv}) yields
{\allowdisplaybreaks \begin{align*}
D_{2}(1,q^3,q^8;q^{4})
&=-q\frac{(q^8;q^8)_{\infty}^3}{\Theta(q;q^4)\Theta(q^8;q^{16})}\frac{1}{\Theta(-q^{4};q^8)}\\
&\qquad \cdot \Big [q^{2}
\frac{\Theta(-q;q^8)\Theta(q^{2};q^{16})}{\Theta(q^3;q^8)}
 + \frac{\Theta(-q^{3};q^8)\Theta(q^{6};q^{16})}{\Theta(q;q^8)}\Big ]\\
&=-q\frac{(q^8;q^8)_{\infty}^3}{\Theta(q;q^4)\Theta(q^8;q^{16})}\frac{1}{\Theta(-q^{4};q^8)}
\frac{1}{\Theta(q;q^8)\Theta(q^3;q^8)}\\
& \qquad \cdot \Big [q^{2}
\Theta(q;q^8)\Theta(-q;q^8)\Theta(q^{2};q^{16})
 + \Theta(q^3;q^8)\Theta(-q^{3};q^8)\Theta(q^{6};q^{16})\Big ].
 \end{align*}}%
 Using (\ref{equation:theta-mod}), (\ref{equation:theta-roots}), and then (\ref{equation:H1Thm1.1}) gives
 \begin{align*}
 D_{2}(1,q^3,q^8;q^{4})
&=-q\frac{(q^8;q^8)_{\infty}^3}{\Theta(q;q^4)\Theta(q^8;q^{16})}\frac{1}{\Theta(-q^{4};q^8)}
\frac{1}{\Theta(q;q^4)}\frac{\Theta_{4}}{\Theta_{8}^2}\frac{\Theta_{8}^2}{\Theta_{16}}\\
& \qquad \cdot \Big [q^{2}
\Theta(q^{2};q^{16})^2
 + \Theta(q^{6};q^{16})^2\Big ]\\
&=-q\frac{(q^8;q^8)_{\infty}^3}{\Theta(q;q^4)^2\Theta(q^8;q^{16})}\frac{1}{\Theta(-q^{4};q^8)}
\frac{\Theta_{4}}{\Theta_{16}}
 \cdot \Theta(-q^2;q^8)\Theta(q^4;q^8)\\
&=-q \frac{\Theta_{2}\Theta_{4}^5\Theta_{16}^2}{\Theta_{1}^2\Theta_{8}^5},
\end{align*}
where the last line follows from elementary product rearragements.

We prove (\ref{equation:D2-B-id2}).  We substitute into (\ref{equation:msplit2}) where we use the form of the single-quotient evaluation \cite[Corollary 3.7]{HM}.  We then have
\begin{equation*}
D_{2}(1,q^{3},q^{12};q^{4})
=-\frac{\Theta_{8}\Theta_{16}\Theta(-q^{6};q^{4})\Theta(-q^{9};q^{4})}
{\Theta(q^3;q^{4})\Theta(q^{12};q^{16})\Theta(-q^{16};q^{8})}.
\end{equation*}
Using (\ref{equation:theta-elliptic}), (\ref{equation:theta-inv}), and elementary product rearrangements yields
\begin{align*}
D_{2}(1,q^{3},q^{12};q^{4})
&=-\frac{1}{2}\frac{\Theta_{2}\Theta_{4}^3\Theta_{8}}{\Theta_{1}^2\Theta_{16}^2}.\qedhere
\end{align*}
\end{proof}

\begin{proof}[Proof of Proposition \ref{proposition:3rho-ids}] 

We prove (\ref{equation:D2-rho-id1}).  Substituting into (\ref{equation:msplit2}) and using (\ref{equation:theta-elliptic}) yields
\begin{align*}
D_{2}(1,q,q^{12};q^{6})
&=-\frac{q(q^{12};q^{12})_{\infty}^3}{\Theta(q;q^6)\Theta(q^{12};q^{24})}\frac{1}{\Theta(-q^{6};q^{12})}\\
&\qquad \cdot \Big [
\frac{\Theta(-q^{7};q^{12})\Theta(q^{10};q^{24})}{\Theta(q;q^{12})}
+q^{4} \frac{\Theta(-q;q^{12})\Theta(q^{2};q^{24})}{\Theta(q^7;q^{12})}\Big ].
\end{align*}
Combining fractions and using (\ref{equation:theta-mod}) and (\ref{equation:theta-roots}) produces
{\allowdisplaybreaks
\begin{align*}
D_{2}(1,q,q^{12};q^{6})
&=-\frac{q(q^{12};q^{12})_{\infty}^3}{\Theta(q;q^6)\Theta(q^{12};q^{24})}\frac{1}{\Theta(-q^{6};q^{12})}
\frac{1}{\Theta(q;q^{6})}\frac{\Theta_{6}}{\Theta_{12}^2}\frac{\Theta_{12}^2}{\Theta_{24}}\\
&\qquad \cdot \Big [
\Theta(q^{14};q^{24})\Theta(q^{10};q^{24})
 +q^{4}\Theta(q^2;q^{24})\Theta(q^{2};q^{24})\Big ].
 \end{align*}
Identity (\ref{equation:H1Thm1.1}) and elementary product rearrangements gives the result
\begin{align*}
D_{2}(1,q,q^{12};q^{6})
&=-\frac{q(q^{12};q^{12})_{\infty}^3}{\Theta(q;q^6)^2\Theta(q^{12};q^{24})}\frac{1}{\Theta(-q^{6};q^{12})}
\frac{\Theta_{6}}{\Theta_{24}}
 \cdot \Theta(-q^4;q^{12})\Theta(q^6;q^{12})\\
&=-q \frac{\Theta_{2}^2\Theta_{3}^2\Theta_{6}\Theta_{8}\Theta_{24}}{\Theta_{1}^2\Theta_{4}\Theta_{12}^3}.
\end{align*}}%

We prove (\ref{equation:D2-rho-id2}).   Substituting into (\ref{equation:msplit2}) and using (\ref{equation:theta-elliptic}) gives
 \begin{align*}
D_2(1,q,q^{18};q^{6})
&=\frac{(q^{12};q^{12})_{\infty}^3}{\Theta(q;q^6)\Theta(q^{18};q^{24})}\frac{1}{\Theta(-1;q^{12})}\\
&\qquad \cdot \Big [
-\frac{\Theta(-q;q^{12})\Theta(q^{16};q^{24})}{\Theta(q;q^{12})}
 + q\frac{\Theta(-q^{7};q^{12})\Theta(q^{4};q^{24})}{\Theta(q^7;q^{12})}\Big ].
 \end{align*}
 Combining fractions and using (\ref{equation:theta-mod}) and (\ref{equation:theta-roots}) yields
 \begin{align*}
D_2(1,q,q^{18};q^{6})
&=-\frac{(q^{12};q^{12})_{\infty}^3}{\Theta(q;q^6)\Theta(q^{18};q^{24})}\frac{1}{\Theta(-1;q^{12})}
\frac{1}{\Theta(q;q^{12})\Theta(q^7;q^{12})}\\
&\qquad \cdot \Big [
\Theta(q^7;q^{12})\Theta(-q;q^{12})\Theta(q^{16};q^{24})
 - q\Theta(q;q^{12})\Theta(-q^{7};q^{12})\Theta(q^{4};q^{24})\Big ]\\
&=-\frac{(q^{12};q^{12})_{\infty}^3}{\Theta(q;q^6)^2\Theta(q^{18};q^{24})}\frac{1}{\Theta(-1;q^{12})}
\frac{\Theta_{6}}{\Theta_{12}^2}\frac{\Theta_{24}}{\Theta_{12}^2}\\
&\qquad \cdot \Big [
\Theta(q^7;q^{12})\Theta(-q;q^{12})\Theta(-q^{8};q^{12})\Theta(q^{8};q^{12})\\
&\qquad \qquad  - q\Theta(q;q^{12})\Theta(-q^{7};q^{12})\Theta(q^{2};q^{12})\Theta(-q^{2};q^{12})\Big ].
\end{align*}
The result follows from the Weierstrass relation (\ref{equation:3termWeier}) with the substitutions $q\to q^{12}$, $(a,b,c,d)\to (iq^6,iq^4,-iq^3,-iq^2)$ and elementary product rearrangements:
 \begin{align*}
D_2(1,q,q^{18};q^{6})
&=-\frac{(q^{12};q^{12})_{\infty}^3}{\Theta(q;q^6)^2\Theta(q^{18};q^{24})}\frac{1}{\Theta(-1;q^{12})}
\frac{\Theta_{6}}{\Theta_{12}^2}\frac{\Theta_{24}}{\Theta_{12}^2}\\
&\qquad \cdot 
\Theta(-q^2;q^{12})\Theta(q^6;q^{12})\Theta(q^{3};q^{12})\Theta(-q^{3};q^{12})\\
&=-\frac{1}{2}\frac{\Theta_{2}\Theta_{3}^2\Theta_{4}^2}
{\Theta_{1}^2\Theta_{8}\Theta_{24}}.\qedhere
\end{align*}

\end{proof}

\begin{proof}[Proof of Proposition  \ref{proposition:3lambda-ids}]  
We prove (\ref{equation:D2-lambda-id1}).  Substituting into (\ref{equation:msplit2}), using (\ref{equation:theta-elliptic}), and combining fractions gives
\begin{align*}
D_2(1,-q^2,q^{12};q^{6})
&=\frac{q^{2}(q^{12};q^{12})_{\infty}^3}{\Theta(-q^{2};q^6)\Theta(q^{12};q^{24})}\frac{1}{\Theta(-q^{6};q^{12})}
\frac{1}{\Theta(-q^{2};q^{12})\Theta(-q^8;q^{12})}\\
&\qquad \cdot \Big [
\Theta(-q^8;q^{12})\Theta(q^{8};q^{12})\Theta(q^{8};q^{24})
 +q^{2} \Theta(-q^{2};q^{12})\Theta(q^{2};q^{12})\Theta(q^{4};q^{24})\Big ].
\end{align*}
Employing (\ref{equation:theta-mod}) and (\ref{equation:theta-roots}) gives
\begin{align*}
D_2(1,-q^2,q^{12};q^{6})
&=\frac{q^{2}(q^{12};q^{12})_{\infty}^3}{\Theta(-q^{2};q^6)^2\Theta(q^{12};q^{24})}\frac{1}{\Theta(-q^{6};q^{12})}
\frac{\Theta_{6}}{\Theta_{12}^2}\frac{\Theta_{12}^2}{\Theta_{24}}\\
&\qquad \cdot \Big [
\Theta(q^{16};q^{24})^2\
 +q^{2} \Theta(q^{4};q^{24})^2\Big ].
 \end{align*}
 Identity (\ref{equation:H1Thm1.1}) and product rearrangements yields
 \begin{align*}
D_2(1,-q^2,q^{12};q^{6})
 &=\frac{q^{2}(q^{12};q^{12})_{\infty}^3}{\Theta(-q^{2};q^6)^2\Theta(q^{12};q^{24})}\frac{\Theta(-q^{2};q^{12})\Theta(q^{6};q^{12})}{\Theta(-q^{6};q^{12})}
\frac{\Theta_{6}}{\Theta_{24}} 
=q^2\frac{\Theta_{2}\Theta_{6}^2\Theta_{24}^3}{\Theta_{8}\Theta_{12}^4}.
\end{align*}

We prove (\ref{equation:D2-lambda-id2}).  Substituting into (\ref{equation:msplit2}), using (\ref{equation:theta-elliptic}), and combining fractions gives
\begin{align*}
D_2(1,-q,q^{12};q^{6})
&=\frac{q(q^{12};q^{12})_{\infty}^3}{\Theta(-q;q^6)\Theta(q^{12};q^{24})}\frac{1}{\Theta(-q^{6};q^{12})}
\frac{1}{\Theta(-q;q^{12})\Theta(-q^7;q^{12})}\\
&\qquad \cdot \Big [
\Theta(-q^7;q^{12})\Theta(q^{7};q^{12})\Theta(q^{10};q^{24})\\
&\qquad \qquad   +q^{4} \Theta(-q;q^{12})\Theta(q;q^{12})\Theta(q^{2};q^{24})\Big ].
\end{align*}
Again using (\ref{equation:theta-mod}) and (\ref{equation:theta-roots}) gives
\begin{align*}
D_2(1,-q,q^{12};q^{6})
&=\frac{q(q^{12};q^{12})_{\infty}^3}{\Theta(-q;q^6)^2\Theta(q^{12};q^{24})}\frac{1}{\Theta(-q^{6};q^{12})}
\frac{\Theta_{6}}{\Theta_{12}^2}\frac{\Theta_{12}^2}{\Theta_{24}}\\
&\qquad  \cdot \Big [
\Theta(q^{14};q^{24})^2 +q^{4} \Theta(q^2;q^{24})^2\Big ].
 \end{align*}
 Identity (\ref{equation:H1Thm1.1}) and product rearrangements give the result:
{\allowdisplaybreaks  \begin{align*}
 D_2(1,-q,q^{12};q^{6})
&=\frac{q(q^{12};q^{12})_{\infty}^3}{\Theta(-q;q^6)^2\Theta(q^{12};q^{24})}\frac{\Theta(-q^4;q^{12})\Theta(q^6;q^{12})}{\Theta(-q^{6};q^{12})}
\frac{\Theta_{6}}{\Theta_{24}} 
=q\frac{\Theta_{1}^2\Theta_{4}\Theta_{6}^7\Theta_{8}\Theta_{24}}{\Theta_{2}^4\Theta_{3}^2\Theta_{12}^5}.
\end{align*}}%

We prove (\ref{equation:D2-lambda-id3}). Using (\ref{equation:msplit2}) with (\ref{equation:theta-elliptic}) yields
\begin{align*}
D_2(1,-q^2,-1;q^{6})
&=-\frac{(q^{12};q^{12})_{\infty}^3}{\Theta(-q^2;q^6)\Theta(-1;q^{24})}\frac{1}{\Theta(q^6;q^{12})}\\
&\qquad \cdot \Big [
\frac{\Theta(-q^8;q^{12})\Theta(-q^4;q^{24})}{\Theta(-q^2;q^{12})}
 + \frac{\Theta(-q^{2};q^{12})\Theta(-q^{16};q^{24})}{\Theta(-q^8;q^{12})}\Big ].
 \end{align*}
We want to write the expression in brackets as a single quotient of theta functions.  In particular, we want to show
\begin{equation*}
\frac{\Theta(-q^8;q^{12})\Theta(-q^4;q^{24})}{\Theta(-q^2;q^{12})}
 + \frac{\Theta(-q^{2};q^{12})\Theta(-q^{16};q^{24})}{\Theta(-q^8;q^{12})}
  =2\frac{\Theta_{6}^2\Theta_{16}^2}{\Theta_{4}\Theta_{8}\Theta_{12}}.
\end{equation*} 
Rewriting the above expression using (\ref{equation:theta-mod}), we have
 \begin{equation*}
 \frac{\Theta_{2}\Theta_{8}^4\Theta_{12}^4\Theta_{48}}{\Theta_{4}^4\Theta_{6}\Theta_{16}\Theta_{24}^3}
 +\frac{\Theta_{4}^3\Theta_{6}\Theta_{16}\Theta_{24}^4}{\Theta_{2}\Theta_{8}^3\Theta_{12}^3\Theta_{48}}
 -2\frac{\Theta_{6}^2\Theta_{16}^2}{\Theta_{4}\Theta_{8}\Theta_{12}}=0,
 \end{equation*}
but this is just (\ref{equation:D2-lambda-id3-thetaID}).  After product rearrangements, it follows that
\begin{align*}
D_2(1,-q^2,-1;q^{6})
 &=-\frac{(q^{12};q^{12})_{\infty}^3}{\Theta(-q^2;q^6)\Theta(-1;q^{24})}\frac{1}{\Theta(q^6;q^{12})}
  \cdot 2\frac{\Theta_{6}^2\Theta_{16}^2}{\Theta_{4}\Theta_{8}\Theta_{12}}
  =-\frac{\Theta_{2}\Theta_{12}^4\Theta_{16}^2\Theta_{24}}{\Theta_{4}^2\Theta_{6}^2\Theta_{8}\Theta_{48}^2}.\qedhere
\end{align*}

\end{proof}

\begin{proof}[Proof of Proposition  \ref{proposition:3Rphi-ids}]   We prove (\ref{equation:D2-Rphi-id1}).  Using (\ref{equation:msplit2}) and simplifying with (\ref{equation:theta-elliptic}), we notice that the two theta quotients are identical.  So we can add them giving
\begin{align*}
D_{2}(1,q,q^4;q^2)
&=-2\frac{q(q^4;q^4)_{\infty}^3}{\Theta(q;q^2)\Theta(q^4;q^8)}\frac{1}{\Theta(-q^2;q^4)}
 \cdot\frac{\Theta(-q;q^4)\Theta(q^{2};q^{8})}{\Theta(q;q^4)}
 =-2q\frac{\Theta_{2}^7\Theta_{8}^4}{\Theta_{1}^4\Theta_{4}^6}.
\end{align*}

We prove (\ref{equation:D2-Rphi-id2}).  Again using (\ref{equation:msplit2}) and simplifying with (\ref{equation:theta-elliptic}), we now notice that one of the two theta quotients vanishes.  This gives
\begin{align*}
D_{2}(1,q,q^{-2};q^2)
&=-\frac{(q^4;q^4)_{\infty}^3}{\Theta(q;q^2)\Theta(q^{2};q^8)}\frac{1}{\Theta(-1;q^4)}
 \cdot 
\frac{\Theta(-q;q^4)\Theta(q^4;q^{8})}{\Theta(q;q^4)}
=-\frac{1}{2}\frac{\Theta_{2}^3\Theta_{4}^6}{\Theta_{1}^4\Theta_{8}^4}.\qedhere
\end{align*}
\end{proof}

\begin{proof}[Proof of Proposition  \ref{proposition:3Rxi-ids}]  We prove (\ref{equation:D2-Rxi-id1})  Using (\ref{equation:msplit2}), we see that one of the two theta quotients vanishes.  Simplifying with (\ref{equation:theta-elliptic}) then gives
\begin{equation*}
D_2(1,-1,q^2;q)
=\frac{(q^2;q^2)_{\infty}^3}{\Theta(-1;q)\Theta(q^2;q^4)}
 \cdot 
\frac{\Theta(q;q^2)\Theta(q^{2};q^{4})}{\Theta(-q;q^2)\Theta(-1;q^2)}
=\frac{1}{4}\frac{\Theta_{1}^5}{\Theta_{2}^4}.
\end{equation*}

We prove (\ref{equation:D2-Rxi-id2})  Again using (\ref{equation:msplit2}), we see that one of the two theta quotients vanishes.  Simplifying with (\ref{equation:theta-elliptic}) then yields
\begin{equation*}
D_2(1,-1,q^{-1};q)
=-\frac{(q^2;q^2)_{\infty}^3}{\Theta(-1;q)\Theta(q;q^4)}
\frac{\Theta(q;q^2)\Theta(q^3;q^{4})}{\Theta(-1;q^2)\Theta(-q;q^2)}
=-\frac{1}{4}\frac{\Theta_{1}^5}{\Theta_{2}^4}.\qedhere
\end{equation*}

\end{proof}

\begin{proof}[Proof of Proposition  \ref{proposition:4psi-ids}].   We prove (\ref{equation:D2-psi-id1}).  Substituting into  (\ref{equation:msplit2}) and simplifying gives
\begin{align*}
D_2(1,-q,-1;q^{3})
&=-\frac{(q^6;q^6)_{\infty}^3}{\Theta(-q;q^3)\Theta(-1;q^{12})}\frac{1}{\Theta(q^3;q^6)}\\
&\qquad \cdot \Big [
\frac{\Theta(-q^4;q^6)\Theta(-q^2;q^{12})}{\Theta(-q;q^6)}
 + \frac{\Theta(-q;q^6)\Theta(-q^8;q^{12})}{\Theta(-q^4;q^6)}\Big ].
\end{align*}
Hence we need to prove
\begin{align*}
-\frac{(q^6;q^6)_{\infty}^3}{\Theta(-q;q^3)\Theta(-1;q^{12})}\frac{1}{\Theta(q^3;q^6)}
 \cdot \Big [
\frac{\Theta(-q^4;q^6)\Theta(-q^2;q^{12})}{\Theta(-q;q^6)}
& + \frac{\Theta(-q;q^6)\Theta(-q^8;q^{12})}{\Theta(-q^4;q^6)}\Big ]\\
&\qquad  =-\frac{\Theta_{1}\Theta_{6}^4\Theta_{8}^2\Theta_{12}}
{\Theta_{2}^2\Theta_{3}^2\Theta_{4}\Theta_{24}^2}.
\end{align*}
Using (\ref{equation:theta-mod}) and simplifying, we see that the above is equivalent to 
\begin{equation*}
-\frac{1}{2}\frac{\Theta_{1}^2\Theta_{4}^4\Theta_{6}^9}{\Theta_{2}^5\Theta_{3}^5\Theta_{8}\Theta_{12}^2\Theta_{24}}
-\frac{1}{2}\frac{\Theta_{2}^2\Theta_{6}^2\Theta_{8}\Theta_{12}^5}{\Theta_{3}^3\Theta_{4}^3\Theta_{24}^3}
=-\frac{\Theta_{1}\Theta_{6}^4\Theta_{8}^2\Theta_{12}}
{\Theta_{2}^2\Theta_{3}^2\Theta_{4}\Theta_{24}^2},
\end{equation*}
but this is just (\ref{equation:equation:D2-psi-id1-thetaID}).  For the proof of (\ref{equation:D2-psi-id2}), we  
use (\ref{equation:msplit2}), to obtain
\begin{align*}
D_2(1,-q,-q^{6};q^3)
&=\frac{q(q^6;q^6)_{\infty}^3}{\Theta(-q;q^3)\Theta(-q^6;q^{12})}\frac{1}{\Theta(q^{3};q^6)}\\
&\qquad \cdot \Big [
\frac{\Theta(-q^{4};q^6)\Theta(-q^{4};q^{12})}{\Theta(-q;q^6)}
 + q\frac{\Theta(-q;q^6)\Theta(-q^2;q^{12})}{\Theta(-q^4;q^6)}\Big ].
\end{align*}
So we need to show
\begin{align*}
\frac{q(q^6;q^6)_{\infty}^3}{\Theta(-q;q^3)\Theta(-q^6;q^{12})}\frac{1}{\Theta(q^{3};q^6)}
& \cdot \Big [
\frac{\Theta(-q^{4};q^6)\Theta(-q^{4};q^{12})}{\Theta(-q;q^6)}
 + q\frac{\Theta(-q;q^6)\Theta(-q^2;q^{12})}{\Theta(-q^4;q^6)}\Big ]\\
 &\qquad =q \frac{\Theta_{1}\Theta_{4}^5\Theta_{6}^6\Theta_{24}^2}
{\Theta_{2}^4\Theta_{3}^2\Theta_{8}^2\Theta_{12}^5},
\end{align*}
which upon using (\ref{equation:theta-mod}) and simplifying is equivalent to 
\begin{equation}
\frac{\Theta_{1}^2\Theta_{4}\Theta_{6}^{10}\Theta_{8}\Theta_{24}}{\Theta_{2}^4\Theta_{3}^5\Theta_{12}^5}
+q\frac{\Theta_{2}\Theta_{6}^5\Theta_{24}^3}{\Theta_{3}^3\Theta_{8}\Theta_{12}^4}
= \frac{\Theta_{1}\Theta_{4}^5\Theta_{6}^6\Theta_{24}^2}
{\Theta_{2}^4\Theta_{3}^2\Theta_{8}^2\Theta_{12}^5},
\end{equation}
but this is just (\ref{equation:equation:D2-psi-id2-thetaID}).
\end{proof}

\section{Proof of identity (\ref{equation:RLN6-A})}\label{section:id0}

We rewrite left-hand side of identity (\ref{equation:RLN6-A}).  We first recall the Appell function forms (\ref{equation:6th-phi(q)}) and (\ref{equation:6th-psi(q)}), and we then use properties (\ref{equation:mxqz-fnq-newz}) and (\ref{equation:changing-z}) to obtain
{\allowdisplaybreaks \begin{align*}
\phi(q^9)&-\psi(q)-q^{-3}\psi(q^9)\\
&=2m(q^9,-1;q^{27})-m(1,-q;q^3)-q^{-3}m(1,-q^{9};q^{27})\\
&=m(q^9,-1;q^{27})+m(q^9,-q^{-9};q^{27})-m(1,-q;q^3)-q^{-3}m(1,-q^{9};q^{27})\\
&=m(q^9,-q^{9};q^{27})+\frac{\Theta_{27}^3\Theta_{9}^2}{\overline{\Theta}_{0,27}\overline{\Theta}_{9,27}^3}\\
&\qquad +m(q^9,-q^{-9};q^{27})-m(1,-q;q^3)-q^{-3}m(1,-q^{9};q^{27}).
\end{align*}}%
Next, property (\ref{equation:mxqz-flip}), Definition (\ref{equation:D3-def}), and Lemma  \ref{lemma:dTerm-id1} yield
{\allowdisplaybreaks \begin{align*}
\phi(q^9)&-\psi(q)-q^{-3}\psi(q^9)\\
&=m(q^9,-q^{9};q^{27})+q^{-9}m(q^{-9},-q^{9};q^{27})-m(1,-q;q^3)-q^{-3}m(1,-q^{9};q^{27})\\
&\quad +\frac{\Theta_{27}^3\Theta_{9}^2}{\overline{\Theta}_{0,27}\overline{\Theta}_{9,27}^3}\\
&=-D_3(1,-q,-q^9;q^{3}) +\frac{\Theta_{27}^3\Theta_{9}^2}{\overline{\Theta}_{0,27}\overline{\Theta}_{9,27}^3}.
\end{align*}%
The result follows from identity (\ref{equation:finalTheta-id1}).

\section{Proof of identity (\ref{equation:newSixth-1})}\label{section:id1}
We take a slightly different approach.  We use Definition (\ref{equation:D3-def}), properties (\ref{equation:mxqz-flip}) and (\ref{equation:changing-z}), and Appell function forms (\ref{equation:6th-rho(q)}) and (\ref{equation:6th-sigma(q)}) to obtain
{\allowdisplaybreaks \begin{align*}
D_3(1,q,q^9;q^6)
&=m(1,q;q^6)-m(q^{18},q^{9};q^{54})+q^{-6}m(1,q^{9};q^{54})-q^{-18}m(q^{-18},q^{9};q^{54})\\
&=m(1,q;q^6)-m(q^{18},q^{9};q^{54})+q^{-6}m(1,q^{9};q^{54})-m(q^{18},q^{-9};q^{54})\\
&=m(1,q;q^6)-2m(q^{18},q^{9};q^{54})+q^{-6}m(1,q^{9};q^{54})
-\frac{\Theta_{54}^3\Theta_{18,54}^2}{\Theta_{9,54}^3\Theta_{27,54}}\\
&=-q\rho(q)+2\sigma(q^9)-q^3\rho(q^9)-\frac{\Theta_{54}^3\Theta_{18,54}^2}{\Theta_{9,54}^3\Theta_{27,54}}.
\end{align*}}%
Rearranging terms, we have
\begin{align*}
q\rho(q)-2\sigma(q^9)+q^3\rho(q^9)
=-D_3(1,q,q^9;q^6)-\frac{\Theta_{54}^3\Theta_{18,54}^2}{\Theta_{9,54}^3\Theta_{27,54}}.
\end{align*}
The result follows from identity (\ref{equation:finalTheta-id2}).

\section{Proof of identity (\ref{equation:newSixth-2})}\label{section:id2}

Using Definition (\ref{equation:D3-def}), property (\ref{equation:mxqz-flip}), and then property (\ref{equation:changing-z}) twice produces 
{\allowdisplaybreaks \begin{align*}
D_3&(1,-q^2,-q^{27};q^6)\\
&=m(1,-q^2;q^6)-m(q^{18},-q^{27};q^{54})+q^{-6}m(1,-q^{27};q^{54})-q^{-18}m(q^{-18},-q^{27};q^{54})\\
&=m(1,-q^2;q^6)-m(q^{18},-q^{27};q^{54})+q^{-6}m(1,-q^{27};q^{54})-m(q^{18},-q^{-27};q^{54})\\
&=m(1,-q^2;q^6)-2m(q^{18},-q^{27};q^{54})+q^{-6}m(1,-q^{18};q^{54})
-q^{12}\frac{\Theta_{54}^3\Theta_{9,54}^2}{\overline{\Theta}_{27,54}^2\overline{\Theta}_{18,54}^2}.
\end{align*}}%
Likewise, using Definition (\ref{equation:D3-def}), property (\ref{equation:mxqz-flip}), and then property (\ref{equation:changing-z}) yields 
\begin{align*}
D_3&(1,-q,-1;q^6)\\
&=m(1,-q;q^6)-m(q^{18},-1;q^{54})+q^{-6}m(1,-1;q^{54})-q^{-18}m(q^{-18},-1;q^{54})\\
&=m(1,-q;q^6)-2m(q^{18},-1;q^{54})+q^{-6}m(1,-1;q^{54})\\
&=m(1,-q;q^6)-2m(q^{18},-1;q^{54})+q^{-6}m(1,-q^9;q^{54})
+q^{-6}\frac{\Theta_{54}^3\Theta_{9,54}^2}{\overline{\Theta}_{0,54}^2\overline{\Theta}_{9,54}^2}.
\end{align*}
Summing the above two expressions and rearranging terms yields
\begin{align*}
q&\lambda(q)+q^{3}\lambda(q^9)-2\mu(q^9)\\
&= D_3(1,-q^2,-q^{27};q^6)+D_3(1,-q,-1;q^6)
+q^{12}\frac{\Theta_{54}^3\Theta_{9,54}^2}{\overline{\Theta}_{27,54}^2\overline{\Theta}_{18,54}^2}
-q^{-6}\frac{\Theta_{54}^3\Theta_{9,54}^2}{\overline{\Theta}_{0,54}^2\overline{\Theta}_{9,54}^2}.
\end{align*}
The result follows by (\ref{equation:finalTheta-id3}).

\section{Proof of identity (\ref{equation:newSixth-3})}\label{section:id3}
Using Definition (\ref{equation:D3-def}), property (\ref{equation:mxqz-flip}), property (\ref{equation:changing-z}), and then the Appell function forms (\ref{equation:phibar}) and (\ref{equation:psibar}) gives
\begin{align*}
D_{3}&(1,q,-1;q^3)\\
&=m(1,q;q^3)-m(q^{9},-1;q^{27})+q^{-3}m(1,-1;q^{27})-q^{-9}m(q^{-9},-1;q^{27})\\
&=m(1,q;q^3)-2m(q^{9},-1;q^{27})+q^{-3}m(1,-1;q^{27})\\
&=m(1,q;q^3)
-2m(q^{9},-q^9;q^{27})-2\frac{\Theta_{27}^3\overline{\Theta}_{9,27}^2}{\Theta_{9}^2\overline{\Theta}_{0,27}\overline{\Theta}_{9,27}}\\
& \qquad \qquad +q^{-3}m(1,q^9;q^{27})+q^{-3}\frac{\Theta_{27}^3\overline{\Theta}_{9,27}^2}{\Theta_{9}^2\overline{\Theta}_{0,27}^2}\\
&=-2\psi_{\_}(q)+2q\frac{\Theta_{6}^3}{2\Theta_{1}\Theta_{2}}
+2\phi_{\_}(q^{9})+2q^{9}\frac{\overline{\Theta}_{27,108}^3}{\Theta_{9}\overline{\Theta}_{9,36}}
-2\frac{\Theta_{27}^3\overline{\Theta}_{9,27}^2}{\Theta_{9}^2\overline{\Theta}_{0,27}\overline{\Theta}_{9,27}}\\
& \qquad \qquad -2q^{-3}\psi_{\_}(q^{9})+2q^{6}\frac{\Theta_{54}^3}{2\Theta_{9}\Theta_{18}}
+q^{-3}\frac{\Theta_{27}^3\overline{\Theta}_{9,27}^2}{\Theta_{9}^2\overline{\Theta}_{0,27}^2}.
\end{align*}
Rearranging terms gives us
{\allowdisplaybreaks \begin{align*}
2\psi\_(q)+2q^{-3}\psi\_(q^9)-2\phi\_(q^9)&=-D_{3}(1,q,-1;q^3)+q\frac{\Theta_{6}^3}{\Theta_{1}\Theta_{2}}
+2q^{9}\frac{\overline{\Theta}_{27,108}^3}{\Theta_{9}\overline{\Theta}_{9,36}}
\\
& \qquad -2\frac{\Theta_{27}^3\overline{\Theta}_{9,27}^2}{\Theta_{9}^2\overline{\Theta}_{0,27}\overline{\Theta}_{9,27}} +q^{6}\frac{\Theta_{54}^3}{\Theta_{9}\Theta_{18}}
+q^{-3}\frac{\Theta_{27}^3\overline{\Theta}_{9,27}^2}{\Theta_{9}^2\overline{\Theta}_{0,27}^2}.
\end{align*}}%
The result follows by (\ref{equation:finalTheta-id4}).

\section{Proofs of Theorems \ref{theorem:3B-ids}--\ref{theorem:4psi-ids}}\label{section:newTheorems}

\begin{proof}[ Proof of Theorem \ref{theorem:3B-ids}]  
We prove (\ref{equation:Bid-1}) and (\ref{equation:Bid-2}).   We have
\begin{align*}
-qB(q)+2A_2(-q^4)
&=D_{2}(1,q^3,q^8;q^{4}),\\
-qB(q)-\frac{1}{2}\mu_2(q^4)
&=D_{2}(1,q^{3},q^{12};q^{4}).
\end{align*}
For the first identity, we use (\ref{equation:2nd-A(q)}) and (\ref{equation:2nd-B(q)}), and we obtain
\begin{equation*}
-qB(q)+2A_2(-q^4)
=m(1,q^3;q^4)-2m(-q^4,q^8;q^{16}).
\end{equation*}
Using (\ref{equation:mxqz-flip}) and then (\ref{equation:mxqz-fnq-z}), we have
\begin{align*}
-qB(q)+2A_2(-q^4)
&=m(1,q^3;q^4)-m(-q^4,q^8;q^{16})+q^{-4}m(-q^{-4},q^{-8};q^{16})\\
&=m(1,q^3;q^4)-m(-q^4,q^8;q^{16})+q^{-4}m(-q^{-4},q^{8};q^{16})\\
&=D_{2}(1,q^3,q^8;q^{4}).
\end{align*}
For the second identity, we recall (\ref{equation:2nd-B(q)}) and (\ref{equation:2nd-mu(q)}) to obtain
\begin{equation*}
-qB(q)-\frac{1}{2}\mu_2(q^4)
=m(1,q^3;q^4)-m(-q^4,-1;q^{16})-m(-q^4,q^{4};q^{16}).
\end{equation*}
Applying (\ref{equation:mxqz-flip}) followed by (\ref{equation:mxqz-fnq-z}) to the third summand yields
\begin{equation*}
-qB(q)-\frac{1}{2}\mu_2(q^4)
=m(1,q^3;q^4)-m(-q^4,-1;q^{16})+q^{-4}m(-q^4,q^{12};q^{16}).
\end{equation*}
Applying (\ref{equation:mxqz-fnq-newz}) followed by (\ref{equation:mxqz-fnq-z}) to the second summand yields
\begin{align*}
-qB(q)-\frac{1}{2}\mu_2(q^4)
&=m(1,q^3;q^4)-m(-q^4,q^{12};q^{16})+q^{-4}m(-q^4,q^{12};q^{16})\\
&=D_{2}(1,q^{3},q^{12};q^{4}).
\end{align*}
Identities (\ref{equation:Bid-1}) and (\ref{equation:Bid-2}) then follow from Proposition \ref{proposition:3B-ids}.  To prove (\ref{equation:Bid-3}), we add (\ref{equation:Bid-1}) and (\ref{equation:Bid-2}) to obtain
\begin{align*}
-2qB(q)-\frac{1}{2}\mu_2(q^4)+2A_2(-q^4)
&=-q\frac{\Theta_{2}\Theta_{4}^5\Theta_{16}^2}{\Theta_{1}^2\Theta_{8}^5}
-\frac{1}{2}\frac{\Theta_{2}\Theta_{4}^3\Theta_{8}}{\Theta_{1}^2\Theta_{16}^2}\\
&=-\frac{1}{2}\frac{\Theta_{2}\Theta_{4}^5}{\Theta_{1}^2\Theta_{8}^4}
\left ( 2q\cdot\frac{\Theta_{16}^2}{\Theta_{8}}
+\frac{\Theta_{8}^5}{\Theta_{4}^2\Theta_{16}^2}\right )\\
&=-\frac{1}{2}\frac{\Theta_{2}\Theta_{4}^5}{\Theta_{1}^2\Theta_{8}^4}
\left ( q\cdot \overline{\Theta}_{0,8}
+\overline{\Theta}_{4,8}\right).
\end{align*}
Using (\ref{equation:theta-split}) gives
\begin{equation*}
-2qB(q)-\frac{1}{2}\mu_2(q^4)+2A_2(-q^4)
=-\frac{1}{2}\frac{\Theta_{2}\Theta_{4}^5}{\Theta_{1}^2\Theta_{8}^4}
\cdot \Theta(-q;q^2) 
=-\frac{1}{2}\frac{\Theta_{2}^6\Theta_{4}^3}{\Theta_{1}^4\Theta_{8}^4}.\qedhere
\end{equation*}
\end{proof}

\begin{proof}[ Proof of Theorem \ref{theorem:3rho-ids}] 
We prove (\ref{equation:rho-id-1}) and (\ref{equation:rho-id-2}).  We immediately have
\begin{align*}
-q\rho(q)+2A_2(-q^6)
&=D_{2}(1,q,q^{12};q^{6}),\\
-q\rho(q)-\frac{1}{2}\mu_2(q^6)
&=D_2(1,q,q^{18};q^{6}),
\end{align*}
and then use (\ref{equation:D2-rho-id1}) and (\ref{equation:D2-rho-id2}) respectively.  To prove (\ref{equation:rho-id-3}), we add (\ref{equation:rho-id-1}) and (\ref{equation:rho-id-2}).  This gives
\begin{align*}
2q\rho(q)+\frac{1}{2}\mu_2(q^6)-2A_2(-q^6)
&=\frac{1}{2}\frac{\Theta_{2}\Theta_{3}^2\Theta_{4}^2}
{\Theta_{1}^2\Theta_{8}\Theta_{24}}+
q \frac{\Theta_{2}^2\Theta_{3}^2\Theta_{6}\Theta_{8}\Theta_{24}}
{\Theta_{1}^2\Theta_{4}\Theta_{12}^3}.
\end{align*}
Hence, we want to prove
\begin{equation*}
\frac{1}{2}\frac{\Theta_{2}\Theta_{3}^2\Theta_{4}^2}
{\Theta_{1}^2\Theta_{8}\Theta_{24}}+
q \frac{\Theta_{2}^2\Theta_{3}^2\Theta_{6}\Theta_{8}\Theta_{24}}
{\Theta_{1}^2\Theta_{4}\Theta_{12}^3}
=\frac{1}{2}\frac{\Theta_{2}^6\Theta_{3}^4}
{\Theta_{1}^4\Theta_{4}^2\Theta_{6}\Theta_{12}^2},
\end{equation*}
but this is easily seen to be equivalent to (\ref{equation:rho-id-3-thetaID}).
\end{proof}

\begin{proof}[Proof of Theorem \ref{theorem:3lambda-ids}]
We prove (\ref{equation:lambda-id-1}) and (\ref{equation:lambda-id-2}). We quickly see that
\begin{align*}
\lambda(q)+4A_2(-q^6)
&=D_2(1,-q^2,q^{12};q^{6})+D_2(1,-q,q^{12};q^{6}),\\
q\lambda(q)-\mu_2(q^6)
&=D_2(1,-q^2,-1;q^{6})+D_2(1,-q,-q^{12};q^{6}).
\end{align*}
With (\ref{equation:D2-lambda-id1}) and (\ref{equation:D2-lambda-id2}) in mind, we see that (\ref{equation:lambda-id-1}) has been reduced to showing 
\begin{equation*}
q\frac{\Theta_{1}^2\Theta_{4}\Theta_{6}^7\Theta_{8}\Theta_{24}}{\Theta_{2}^4\Theta_{3}^2\Theta_{12}^5}
+q^2\frac{\Theta_{2}\Theta_{6}^2\Theta_{24}^3}{\Theta_{8}\Theta_{12}^4}
=q\frac{\Theta_{1}\Theta_{3}\Theta_{4}^5\Theta_{6}^3\Theta_{24}^{2}}
{\Theta_{2}^4\Theta_{8}^{2}\Theta_{12}^5},
\end{equation*}
but this is just (\ref{equation:lambda-id-1-thetaID}).  For (\ref{equation:lambda-id-2}), we see that
\begin{align*}
D_2(1,-q,-q^{12};q^{6})
&=\frac{q(q^{12};q^{12})_{\infty}^3}{\Theta(-q;q^6)\Theta(-q^{12};q^{24})}\frac{1}{\Theta(q^{6};q^{12})}\\
&\qquad \cdot \Big [
\frac{\Theta(-q^{7};q^{12})\Theta(-q^{10};q^{24})}{\Theta(-q;q^{12})}
 + q^{4}\frac{\Theta(-q;q^{12})\Theta(-q^{2};q^{24})}{\Theta(-q^7;q^{12})}\Big ].
\end{align*}
With (\ref{equation:D2-lambda-id3}) in mind, we see that showing (\ref{equation:lambda-id-2}) is equivalent to proving
\begin{align*}
&-\frac{\Theta_{2}\Theta_{12}^4\Theta_{16}^2\Theta_{24}}{\Theta_{4}^2\Theta_{6}^2\Theta_{8}\Theta_{48}^2}
 + \frac{q(q^{12};q^{12})_{\infty}^3}{\Theta(-q;q^6)\Theta(-q^{12};q^{24})}\frac{1}{\Theta(q^{6};q^{12})}\\
&\qquad \cdot \Big [
\frac{\Theta(-q^{7};q^{12})\Theta(-q^{10};q^{24})}{\Theta(-q;q^{12})}
 + q^{4}\frac{\Theta(-q;q^{12})\Theta(-q^{2};q^{24})}{\Theta(-q^7;q^{12})}\Big ]
  = -\frac{\Theta_{1}\Theta_{3}\Theta_{6}\Theta_{8}^2\Theta_{12}}
{\Theta_{2}^2\Theta_{4}\Theta_{24}^2},
\end{align*}
which is equivalent to 
\begin{align*}
&-\frac{\Theta_{2}\Theta_{12}^4\Theta_{16}^2\Theta_{24}}{\Theta_{4}^2\Theta_{6}^2\Theta_{8}\Theta_{48}^2}
+q\frac{\Theta_{1}\Theta_{4}\Theta_{12}^5\Theta_{48}}{\Theta_{2}^2\Theta_{3}\Theta_{6}\Theta_{24}^3}\cdot \left ( 
\frac{\Theta_{1,12}\Theta_{20,48}}{\Theta_{5,12}\Theta_{2,24}}
+q^4\cdot 
\frac{\Theta_{5,12}\Theta_{4,48}}{\Theta_{1,12}\Theta_{10,24}}\right ) 
+\frac{\Theta_{1}\Theta_{3}\Theta_{6}\Theta_{8}^2\Theta_{12}}
{\Theta_{2}^2\Theta_{4}\Theta_{24}^2}=0,
\end{align*}
but this is just (\ref{equation:lambda-id-2-thetaID}).

We prove (\ref{equation:lambda-id-3}) by first adding  (\ref{equation:lambda-id-1}) and  (\ref{equation:lambda-id-2}) 
\begin{align*}
q\lambda(q)-\frac{1}{2}\mu_2(q^6)+2A_2(-q^6)
&=\frac{q}{2}\frac{\Theta_{1}\Theta_{3}\Theta_{4}^5\Theta_{6}^3\Theta_{24}^{2}}
{\Theta_{2}^4\Theta_{8}^{2}\Theta_{12}^5}
-\frac{1}{2}\frac{\Theta_{1}\Theta_{3}\Theta_{6}\Theta_{8}^2\Theta_{12}}
{\Theta_{2}^2\Theta_{4}\Theta_{24}^2}.
\end{align*}
Hence we need to show
\begin{equation*}
\frac{q}{2}\frac{\Theta_{1}\Theta_{3}\Theta_{4}^5\Theta_{6}^3\Theta_{24}^{2}}
{\Theta_{2}^4\Theta_{8}^{2}\Theta_{12}^5}
-\frac{1}{2}\frac{\Theta_{1}\Theta_{3}\Theta_{6}\Theta_{8}^2\Theta_{12}}
{\Theta_{2}^2\Theta_{4}\Theta_{24}^2}
=-\frac{1}{2}\frac{\Theta_{1}^2\Theta_{3}^2\Theta_{4}\Theta_{6}^2}
{\Theta_{2}^3\Theta_{12}^3},
\end{equation*}
but this is just (\ref{equation:lambda-id-3-thetaID}).
\end{proof}

\begin{proof} [Proof of Theorem \ref{theorem:3Rphi-ids}]  We prove (\ref{equation:Rphi-id-1}) and (\ref{equation:Rphi-id-2}).  We immediately obtain
\begin{align*}
-2\phi_{R}(q)+2A_{2}(-q^2)
&=D_{2}(1,q,q^4;q^2),\\
-2\phi_{R}(q)-\frac{1}{2}\mu_{2}(q^2)
&=D_{2}(1,q,q^{-2};q^2),
\end{align*}
and then use (\ref{equation:D2-Rphi-id1}) and (\ref{equation:D2-Rphi-id2}).  For identity (\ref{equation:Rphi-id-3}), we first add (\ref{equation:Rphi-id-1}) and (\ref{equation:Rphi-id-2}).  This yields
\begin{equation*}
-4\phi_{R}(q)+\frac{1}{2}\mu_{2}(q^2)+2A_{2}(-q^2)
=-2q\frac{\Theta_{2}^7\Theta_{8}^4}{\Theta_{1}^4\Theta_{4}^6}
-\frac{1}{2}\frac{\Theta_{2}^3\Theta_{4}^6}{\Theta_{1}^3\Theta_{8}^4}.
\end{equation*}
Hence we need to prove
\begin{equation*}
2q\frac{\Theta_{2}^7\Theta_{8}^4}{\Theta_{1}^4\Theta_{4}^6}
+\frac{1}{2}\frac{\Theta_{2}^3\Theta_{4}^6}{\Theta_{1}^4\Theta_{8}^4}
=\frac{1}{2}\frac{\Theta_{2}^{17}}{\Theta_{1}^8\Theta_{4}^8},
\end{equation*}
which is equivalent to (\ref{equation:Rphi-id-3-thetaID}).
\end{proof}

\begin{proof}[Proof of Theorem \ref{theorem:3Rxi-ids}]    We prove (\ref{equation:3xi-id-1}) and (\ref{equation:3xi-id-2}).  We have
\begin{align*}
\frac{1}{2}\xi_{R}(q)+2A_2(-q)
&=D_2(1,-1,q^2;q),\\
\frac{1}{2}\xi_{R}(q)-\frac{1}{2}\mu_(q)
&=D_2(1,-1,q^{-1};q),
\end{align*}
and then use (\ref{equation:D2-Rxi-id1}) and (\ref{equation:D2-Rxi-id2}).  To prove (\ref{equation:3xi-id-3}), we simply add (\ref{equation:3xi-id-1}) and (\ref{equation:3xi-id-2}).
\end{proof}

\begin{proof}[Proof of Theorem  \ref{theorem:4psi-ids}] We first prove (\ref{equation:psi-id-1}) and (\ref{equation:psi-id-2}).    It is straightforward to show
\begin{align*}
\psi(q)-U_0(q^3)
&=D_2(1,-q,-1;q^{3}),\\
\psi(q)+2U_1(q^3)
&=D_2(1,-q,-q^{6};q^3),
\end{align*}
and then we use (\ref{equation:D2-psi-id1}) and (\ref{equation:D2-psi-id2}).   To prove (\ref{equation:psiBar-id-1}) and (\ref{equation:psiBar-id-2}), we first obtain
\begin{align*}
-2\psi_{\_}(q)-U_0(q^3)
&=D_{2}(1,q,-1;q^3)-q\frac{\Theta_{6}^3}{\Theta_{1}\Theta_{2}},\\
-2\psi_{\_}(q)+2U_1(q^3)
&=D_2(1,q,-q^{6};q^{3})-q\frac{\Theta_{6}^3}{\Theta_{1}\Theta_{2}}.
\end{align*}
For (\ref{equation:psiBar-id-1}), we substitute into (\ref{equation:msplit2}) and simplify to obtain
\begin{align*}
D_{2}(1,q,-1;q^3)
&=-\frac{(q^6;q^6)_{\infty}^3}{\Theta(q;q^3)\Theta(-1;q^{12})}\frac{1}{\Theta(q^3;q^6)}\\
&\qquad \cdot \Big [
\frac{\Theta(q^4;q^6)\Theta(-q^2;q^{12})}{\Theta(q;q^6)}
+ \frac{\Theta(q;q^6)\Theta(-q^8;q^{12})}{\Theta(q^4;q^6)}\Big ].
\end{align*}
Hence we need to prove
\begin{align*}
-\frac{(q^6;q^6)_{\infty}^3}{\Theta(q;q^3)\Theta(-1;q^{12})}\frac{1}{\Theta(q^3;q^6)}
& \cdot \Big [
\frac{\Theta(q^4;q^6)\Theta(-q^2;q^{12})}{\Theta(q;q^6)}
+ \frac{\Theta(q;q^6)\Theta(-q^8;q^{12})}{\Theta(q^4;q^6)}\Big ]
  -q\frac{\Theta_{6}^3}{\Theta_{1}\Theta_{2}} \\
&=- \frac{\Theta_{2}\Theta_{4}^{2}\Theta_{6}^3}{\Theta_{1}^2\Theta_{3}\Theta_{8}\Theta_{24}},
\end{align*}
which is equivalent to 
\begin{equation*}
-\frac{1}{2}\frac{\Theta_{2}\Theta_{4}^2\Theta_{6}^3}{\Theta_{1}^2\Theta_{3}\Theta_8\Theta_{24}}
-\frac{1}{2}\frac{\Theta_{6}^6\Theta_{8}\Theta_{12}^3}{\Theta_{2}^2\Theta_{3}^3\Theta_{4}\Theta_{24}^3}
  -q\frac{\Theta_{6}^3}{\Theta_{1}\Theta_{2}} 
=- \frac{\Theta_{2}\Theta_{4}^{2}\Theta_{6}^3}{\Theta_{1}^2\Theta_{3}\Theta_{8}\Theta_{24}},
\end{equation*}
but this is just (\ref{equation:psiBar-id-1-thetaID}).  For (\ref{equation:psiBar-id-2}), we again use (\ref{equation:msplit2}) and simplify to obtain
\begin{align*}
D_2(1,q,-q^{6};q^{3})
&=-\frac{q(q^6;q^6)_{\infty}^3}{\Theta(q;q^3)\Theta(-q^6;q^{12})}\frac{1}{\Theta(q^3;q^6)}\\
&\qquad \cdot \Big [
\frac{\Theta(q^{4};q^6)\Theta(-q^{4};q^{12})}{\Theta(q;q^6)}
 +q \frac{\Theta(q;q^6)\Theta(-q^2;q^{12})}{\Theta(q^4;q^6)}\Big ].
\end{align*}
Hence we need to prove
\begin{align*}
\frac{q(q^6;q^6)_{\infty}^3}{\Theta(q;q^3)\Theta(-q^6;q^{12})}\frac{1}{\Theta(q^3;q^6)}
& \cdot \Big [
\frac{\Theta(q^{4};q^6)\Theta(-q^{4};q^{12})}{\Theta(q;q^6)}
 +q \frac{\Theta(q;q^6)\Theta(-q^2;q^{12})}{\Theta(q^4;q^6)}\Big ]
 +q\frac{\Theta_{6}^3}{\Theta_{1}\Theta_{2}}\\
 &=2q\frac{\Theta_2^2\Theta_6^4\Theta_8\Theta_{24}}{\Theta_{1}^2\Theta_{3}\Theta_{4}\Theta_{12}^3},
\end{align*}
which is equivalent to
\begin{equation*}
\frac{\Theta_{2}^2\Theta_{6}^4\Theta_{8}\Theta_{24}}
{\Theta_{1}^2\Theta_{3}\Theta_{4}\Theta_{12}^3}
+q\frac{\Theta_{4}^2\Theta_{6}^9\Theta_{24}^3}{\Theta_{2}^3\Theta_{3}^3\Theta_{8}\Theta_{12}^6}
 +\frac{\Theta_{6}^3}{\Theta_{1}\Theta_{2}}
 =2\frac{\Theta_2^2\Theta_6^4\Theta_8\Theta_{24}}{\Theta_{1}^2\Theta_{3}\Theta_{4}\Theta_{12}^3},
\end{equation*}
but this is just (\ref{equation:psiBar-id-2-thetaID}).
\end{proof}

\section*{Acknowledgements}
This work was supported by the Theoretical Physics and Mathematics Advancement Foundation BASIS, agreement No. 20-7-1-25-1, and by the Ministry of Science and Higher Education of the Russian Federation (agreement no. 075-15-2022-287).   We would like to thank Frank Garvan and Jeremy Lovejoy for helpful comments and suggestions.  We would also like to thank the referee for their careful effort in reviewing the manuscript and in asking whether or not there are more mock theta function identities explained by Theorem  \ref{theorem:msplit-general-n}.  We would also like to thank the Online Encyclopedia for Integer Sequences for guiding us to work of Cheng, Duncan, and Harvey \cite{CDH}.

\end{document}